\documentclass{amsproc}

\input{macros.tex}

\begin{document}

\title[Random growth models]{Random growth models: shape and convergence rate}


\author{Michael Damron}
\address{School of Mathematics, Georgia Institute of Technology, 686 Cherry St., Atlanta, GA, USA\afterpage{\blankpage}}
\email{mdamron6@gatech.edu}
\thanks{This work was done with support from an NSF CAREER grant.}



\date{\today}

\begin{abstract}
Random growth models are fundamental objects in modern probability theory, have given rise to new mathematics, and have numerous applications, including tumor growth and fluid flow in porous media. In this article, we introduce some of the typical models and the basic analytical questions and properties, like existence of asymptotic shapes, fluctuations of infection times, and relations to particle systems. We then specialize to models built on percolation (first-passage percolation and last-passage percolation) and give a self-contained treatment of the shape theorem, the subadditive ergodic theorem, and conjectured and proven properties of asymptotic shapes. We finish by discussing the rate of convergence to the limit shape, along with definitions of scaling exponents and a sketch of the proof of the KPZ scaling relation.
\end{abstract}

\maketitle


\tableofcontents

\section{Introduction}

\subsection{Some typical growth models}

The typical setting for a random growth model is as follows. We imagine an infection sitting at a vertex $v$ of a connected graph $G = (V,E)$ with vertex set $V$ and edge set $E$. The infection spreads along the edges of the graph according to some random rules, and each vertex of the graph is eventually infected. The infection takes time $T(v,w) < \infty$ to infect a vertex $w$, and at time $t$, a set of vertices
\[
B(t) = \{x \in V : T(v,x) \leq t\}
\]
is infected. \smallskip

\index{Eden model}%
{\bf Eden model.} One of the simplest examples is the Eden model, introduced by Eden \cite{Eden} in '61, and gives a simplified version of cell reproduction. Eden considered the two dimensional square lattice $\mathbb{Z}^2$ with its nearest-neighbor edges, but we could consider any graph, and we will generally take the $d$-dimensional cubic lattice 
\[
\mathbb{Z}^d = \{x = (x_1, \ldots, x_d) : x_i \in \mathbb{Z} \text{ for all }i\}
\]
with the nearest-neighbor edges
\[
\mathcal{E}^d = \{\{x,y\} : x,y \in \mathbb{Z}^d, \|x-y\|_1=1\},
\]
and $\|x\|_1 = \sum_i |x_i|$ is the $\ell^1$-norm. We begin with a cell occupying the site $0 = (0, \ldots, 0)$, so our occupied set at time $0$ is $S_0=\{0\}$. At time $n \geq 1$, we consider the edge boundary of $S_{n-1}$
\[
\partial_e S_{n-1} = \{\{x,y\} \in \mathcal{E}^d : x \in S_{n-1},y \notin S_{n-1}\}
\]
and select some edge $e_n = \{x_n,y_n\}$ uniformly at random from it. Last, put
\[
S_n = S_{n-1} \cup \{x_n,y_n\}.
\]
Thus, in the Eden model, our cluster of cells $S_{n-1}$ replicates on the boundary (uniformly at random) and produces a new cell directly outside of $S_{n-1}$ to form our new cluster $S_n$. It turns out that $\cup_n S_n$ is all of $\mathbb{Z}^d$, and so the infection time from $0$ to $x$ (the first value of $n$ such that $S_n$ contains $x$) is finite. The Eden model can be rephrased in a larger framework, first-passage percolation (FPP), which we will meet soon. As a consequence, one can show a ``shape theorem'' for $S_n$: after proper scaling, the set $S_n$ approaches a limiting shape. Although our selection rule appears not to bias any direction, the limiting shape is expected not to be rotationally invariant (that is, it is expected not to be a Euclidean ball). This statement is proved for high dimensions (see \cite{aspects, AT} and the references in Section~\ref{sec: high_dimensions_FPP} below).\smallskip

\index{diffusion-limited aggregation}%
{\bf Diffusion-limited aggregation.} In the Eden model, we select a boundary vertex $y_n$ by picking a boundary edge uniformly at random. We could change this selection rule, and pick $y_n$ according to some other distribution. In diffusion-limited-aggregation (DLA), introduced in \cite{WS}, we select $y_n$ according to the ``harmonic measure from infinity.'' Roughly speaking, this corresponds to the hitting distribution of a random walk on the boundary of $S_n$ started from a far away point. In this way we again obtain a sequence of growing sets $(S_n)$. Since the DLA model is notoriously hard to analyze, a simplified version, internal DLA (IDLA) was introduced in \cite{MD}, where the selection rule is more straightforward. We begin again with $S_0 = \{0\}$. At time $n$, we run a simple symmetric random walk started at 0, and set $y_n$ (our selected vertex) to be the first vertex adjacent to $S_{n-1}$ that the random walk touches. Precisely, at time $n$, we let $X_1^{(n)}, X_2^{(n)}, \ldots$ be an i.i.d. sequence of random vectors taking values
\[
\mathbb{P}(X_i^{(n)} = \pm e_j) = 1/(2d) \text{ for } i \geq 1,~ j =1, \ldots, d,
\]
where $(e_j)$ are the $j$ standard basis vectors of $\mathbb{R}^d$, set 
\[
Y_0^{(n)} = 0, \text{ and } Y_i^{(n)} = \sum_{k=1}^i X_k^{(n)} \text{ for } i \geq 1,
\]
and let $y_n$ be the first element in the sequence $Y_0^{(n)}, Y_1^{(n)}, Y_2^{(n)}, \ldots$ that is not in $S_{n-1}$. There is no shape theorem proved for DLA, but there is one for IDLA \cite{LBG} (convergence to a Euclidean ball), and even the rate of convergence to the shape is known \cite{JLS}.\smallskip

\index{first-passage percolation (FPP)}%
\index{FPP}%
{\bf First-passage percolation.} FPP can be seen as a generalization of the Eden model, and was introduced \cite{HW} by Hammersley and Welsh in '65. In FPP, the infection spreads across edges according to explicit speeds. We let $(t_e)_{e \in \mathcal{E}^d}$ be a collection of i.i.d. nonnegative random variables. The variable $t_e$ is thought of as the passage time of an edge; that is, the amount of time it takes for an infection to cross the edge. A path $\Gamma$ is a sequence of edges $e_0, \ldots, e_n$ such that each pair $e_i$ and $e_{i+1}$ shares an endpoint, and the passage time of such a $\Gamma$ is $T(\Gamma) = \sum_{e \in \Gamma}t_e$. The infection takes the path of minimal passage time, so we set the infection time, or passage time, from $x$ to $y$, vertices in $\mathbb{Z}^d$, to be
\[
T(x,y) = \inf_{\Gamma : x \to y} T(\Gamma),
\]
where the infimum is over paths $\Gamma$ starting at $x$ and ending at $y$. (It is known that under general assumptions, for instance if $\mathbb{P}(t_e=0)<p_c$, where $p_c$ is the $d$-dimensional bond percolation threshold, then there is a unique minimizing path --- a geodesic --- from $x$ to $y$. However, for some distributions with $\mathbb{P}(t_0=0)=p_c$ in dimensions $d \geq 3$, existence of geodesics is unknown.) In FPP, there is a shape theorem, but the limiting shape depends on the distribution of the $(t_e)$'s. Very little is known about the limit shapes for various distributions, apart from them being convex, compact, and having the symmetries of $\mathbb{Z}^d$. As we will see, it is expected that for most distributions, the 
\index{limit shape}%
limit shape is 
\index{strict convexity}%
strictly convex, and certainly not a polygon, but strict convexity is not proved for any distribution, and there are only some two-dimensional examples of limit shapes that are not polygons. For a recent survey on FPP, see \cite{ADH15}.

If the weights $(t_e)$ has (rate 1) exponential distribution, the evolution of the ball $B(t)$ as $t$ grows can be shown to be exactly the same as the evolution of the sets $(S_n)$ in the Eden model. More precisely, using the ``memoryless property'' of the exponential distribution, one can show that the growth of $B(t)$ is the same as in the following algorithm: begin with $B(0) = \{0\}$ and assign i.i.d. exponential random variables to the edges in the set $\partial_e B(0)$. If $\tau_1$ is the minimum of these variables, and it is assigned to edge $e = \{x,y\}$, then set $B(t) = B(0)$ for $t \in [0,\tau_1)$ and $B(\tau_1) = B(0) \cup\{x,y\}$. Next, generate new i.i.d. exponential random variables assigned to the edges in $\partial_e B(\tau_1)$, and set $\tau_2'$ to be the minimum of these variables, with $\tau_2 = \tau_1 + \tau_2'$. Once again, set $B(t) = B(\tau_1)$ for $t \in [\tau_1,\tau_2)$ and $B(\tau_2) = B(\tau_1) \cup \{x',y'\}$, where $\{x',y'\}$ is the edge in $\partial_e B(\tau_1)$ with minimal weight. We continue, and at each step, we sample i.i.d. exponentials for the boundary edges of our current set, choose the minimal weight edge (of weight $\tau$), and add its endpoints into our set after we wait for time $\tau$. Since the location of the minimum is uniformly distributed on the boundary, the sequence of sets $B(0), B(\tau_1), B(\tau_2), \ldots$ has the same distribution as the sequence of sets in the Eden model. There is no such representation of the DLA models in terms of FPP.\smallskip

\index{last-passage percolation (LPP)}%
\index{LPP}%
{\bf Last-passage percolation.} LPP is a modification of FPP, introduced because of its relationship to the TASEP particle system. The typical setting is $\mathbb{Z}^d$, and one places i.i.d. nonnegative random variables (weights) $(t_v)_{v \in \mathbb{Z}^d}$ on the vertices. A path $\Gamma$ is a sequence of vertices $v_0, \ldots, v_n$ such that $\|v_i-v_{i+1}\|_1 = 1$ for all $i$, and one assigns the passage time $T(\Gamma) = \sum_{v \in \Gamma} t_v$, as in FPP. The difference now is that in LPP we define the passage time between two vertices as the maximal passage time of any path between them. Of course this will generally be infinity unless we restrict ourselves to a finite set of paths, so we consider oriented paths; that is, paths such that all the coordinates of the $v_i$'s are nondecreasing (written $v_i \leq v_{i+1}$). So for any $v \leq w$, we set the infection time $T(v,w)$ to be the maximal passage time of any oriented path rom $v$ to $w$. Once again there is a shape theorem; however, unlike in FPP, the limiting shape is generally compact but not convex. In two dimensions, it is believed that the boundary of the limit shape is the graph of a strictly concave function, but this is again not known. In LPP, however, it is known that the limit shape is not a polygon. For a survey of LPP, see \cite{Martin}.

\subsection{Main questions}

As we mentioned many times in the last section, a fundamental object of study in random growth models is the limit shape. In a certain sense to be described in the next section, we will have almost surely $B(t)/t \to \mathcal{B}$, where $\mathcal{B}$ is the limit shape and $B(t)$ is the set of infected sites at time $t$. Directly related to the limit shape are the following two questions.
\begin{enumerate}
\item {\bf Set of limit shapes.} What does the limit shape look like? Are there explicit descriptions? Is it a Euclidean ball? Is it the $\ell^1$ or $\ell^\infty$ ball? In models (like FPP and LPP) where many limit shapes can arise, what is the collection of all possible limit shapes? What is the dependence in FPP and LPP of the limit shape on the weight distribution?
\item {\bf Convergence to the limit.} What is the convergence rate to the limit shape? Precisely, what is the set of functions $f(t)$ such that one has 
\[
(t-f(t))\mathcal{B} \subset B(t) \subset (t+f(t))\mathcal{B}
\]
for $t$ large?
\end{enumerate}
We will focus here on these two questions, but we mention also some of the questions from the other articles in this proceedings volume.  
\begin{enumerate}
\item {\bf Distributional limits of passage times.} Letting $T(x,y)$ be the infection time of $y$ started at $x$, are there functions $a(x)$ and $b(x)$ such that
\[
\frac{T(0,x) - a(x)}{b(x)} \Rightarrow X
\]
as $\|x\|_1 \to \infty$ for some nondegenerate limiting distribution $X$? In two\hyp{}dimensional FPP/LPP-type models, the answer is expected (and proved in a couple of cases) to be yes, and $X$ should have the 
\index{Tracy-Widom distribution}%
Tracy-Widom distribution from random matrix theory. $b(x)$ should be the order of fluctuations of $T(0,x)$, and embedded in this question is obviously the question: what is the order of fluctuations of $T(0,x)$? In the DLA model, fluctuations are of significantly lower order than in FPP/LPP-type models. Work on these questions has led into concentration of measure, particle systems, integrable probability and exactly solvable systems.
\item {\bf Structure of geodesics.} Optimal infection paths are called geodesics. What is the structure of the set of all geodesics? How different are geodesics from straight lines? Infinite geodesics are infinite paths all whose segments are geodesics. How many infinite geodesics are there? Are there doubly-infinite geodesics? These questions are all related to Busemann functions in metric geometry.
\end{enumerate}

The existence (or nonexistence) of 
\index{geodesic!doubly-infinite}%
doubly-infinite geodesics mentioned above is directly related to the number of ground states of disordered ferromagnetic spin models. We explain the connection to the 
\index{disordered ferromagnet}%
disordered Ising ferromagnet, which is a variant of the usual Ising model from statistical mechanics. We consider dimension two, and define the dual lattice
\[
(\mathbb{Z}_*^2, \mathcal{E}_*^2) = (\mathbb{Z}^2,\mathcal{E}^2) + (1/2,1/2),
\]
which is the usual two-dimensional lattice shifted by $(1/2,1/2)$. Each ``dual edge'' $e^*$ in $\mathcal{E}_*^2$ bisects a unique edge $e$ in $\mathcal{E}^2$, so we say that $e^*$ is the edge dual to $e$. Define a ``spin configuration'' $\sigma = (\sigma_x)_{x \in \mathbb{Z}^2_*}$ to be an assignment of $+1$ and $-1$ to every dual vertex; that is, $\sigma$ is an element of $\{-1,+1\}^{\mathbb{Z}^2_*}$. The interactions between spins are given by ``couplings,'' which are variables assigned to the edges. Accordingly, let $(J_{x,y})_{\{x,y\} \in \mathcal{E}^2_*}$ be a family of independent random variables which are almost surely positive. For any $\sigma$ and any finite $S \subset \mathbb{Z}^2_*$, we define the random energy of $\sigma$ relative to the couplings and the set $S$ as
\[
\mathcal{H}_S(\sigma) = -\sum_{\{x,y\} \in \mathcal{E}^2_*, x \in S} J_{x,y} \sigma_x\sigma_y.
\]

For the standard disordered Ising ferromagnet, the couplings $(J_{x,y})$ are typically chosen to be identically distributed. In this case, it is of great interest to determine the structure and number of ground states for the model. Precisely, a configuration $\sigma$ is called a ground state for the couplings $(J_{x,y})$ if for each configuration $\tilde \sigma$ such that $\tilde \sigma_x \neq \sigma_x$ for only finitely many $x$, one has
\[
\mathcal{H}_S(\sigma) \leq \mathcal{H}_S(\tilde \sigma) \text{ for all finite } S \subset \mathbb{Z}^2_*.
\]
One can think of a ground state as a local minimizer of the energy functional. It is not known how many ground states there are for a given realization of couplings, but it is believed that in this two-dimensional model (and sufficiently low-dimensional analogues), there should be only two almost surely. These two are the all-plus and all-minus states.

If there exists a nonconstant ground state $\sigma$ for couplings $(J_{x,y})$, then we can compare it to the all-plus ground state $\sigma_+$. We claim that $\sigma$ and $\sigma_+$ cannot have any finite regions of disagreement: there is no finite $S$ such that $\sigma_x = -1$ for all $x \in S$ and $\sigma_y = +1$ for all $y \in \partial S$ (here $\partial S$ refers to the set of vertices in $S^c$ with a neighbor in $S$). Indeed, if there were such an $S$, we could apply the energy minimization property to $S$ to find that $\mathcal{H}_S(\sigma)=0$. However, as long as the couplings are continuously distributed, one can argue that almost surely, there are no finite $S$ with $\mathcal{H}_S(\sigma)=0$ for some $\sigma$. This justifies the claim; from it we see that any nonconstant ground state must have a two-sided and circuitless original lattice path of edges whose dual edges $\{x,y\}$ satisfy $\sigma_x \neq \sigma_y$. That is, any nonconstant ground state can be associated to at least one such doubly-infinite path. 

Conversely, one can construct nonconstant ground states if one assumes existence of doubly-infinite geodesics in a related FPP model. Namely, given couplings $(J_{x,y})$ as above, associated to the edges of the dual lattice, one defines a passage-time configuration $(t_e)$ by setting $t_e = J_{x,y}$, where $\{x,y\}$ is the unique edge dual to $e$. Supposing that there is a doubly-infinite geodesic $\Gamma$ (this is a doubly-infinite path each of whose subpaths is an optimizing path for $T$) in this configuration, one can set $\sigma_x = +1$ for all $x$ on one side of $\Gamma$ and $-1$ for all $x$ on the other side of $\Gamma$, and such $\sigma$ will be a nonconstant ground state for $(J_{x,y})$.

The relation between ground states and geodesics allows one to carry results between the two models. For example, if one can rule out existence of doubly-infinite geodesics in FPP models, one can deduce nonexistence of nonconstant ground states for the associated spin models.

\section{A bit of basic probability}

Before we begin, we review some of the notions from probability, and discuss typical limiting behavior, to compare to that which we see in random growth models. The main object of study is a random variable, say $X$, which is formally a function from a probability space to $\mathbb{R}$, but we think of it as a random number. A random variable is characterized by its distribution function $F$, which is a function $F: \mathbb{R} \to \mathbb{R}$ satisfying
\begin{enumerate}
\item (nondecreasing) $F(x) \leq F(y)$ if $x \leq y$,
\item (behavior at infinity) 
\[
\lim_{x \to -\infty} F(x) = 0 \text{ and } \lim_{x \to \infty} F(x) = 1,
\]
\item and (right continuity) $F$ is a right continuous function.
\end{enumerate}
We think of $F(x)$ as being the probability that the random variable $X$ is no bigger than $x$:
\[
F(x) = \mathbb{P}(X \leq x),
\]
and so we can compute the probability that $X$ lies in any given interval as
\[
\mathbb{P}(X \in (a,b]) = F(b) - F(a).
\]
It is a theorem that to every such function $F$ there is a random variable $X$ with these properties.

The multidimensional analogue of a random variable is usually called a random vector. A random vector with values in $\mathbb{R}^n$ is simply an $n$-tuple $X = (X_1, \ldots, X_n)$, each of whose coordinates are random variables (and all these variables are on the same probability space). Similarly to the one-dimensional case, a random vector is characterized by its distribution function, which is a function $F$ on $\mathbb{R}^n$ (representing $F(x) = \mathbb{P}(X \leq x)$, where $X \leq x$ means $X_i \leq x_i$ for all $i$) which is nondecreasing ($F(x) \leq F(y)$ when $x \leq y$), satisfies
\[
\lim_{x_i \to -\infty \text{ for all }i} F(x) = 0 \quad \text{and} \quad \lim_{x_i\to \infty \text{ for all } i} F(x) = 1,
\]
and has the right continuity property $\lim_{x \downarrow y} F(x) = y$ for all $y$. Unlike in the one-dimensional case, an additional positivity property for $F$ is required to guarantee that there is a random vector $X$ with distribution function $F$. See \cite[Section~3.9]{DurrettBook} for more details.

Back to the one-dimensional setting, in many cases of interest, $F$ is actually differentiable, and we write its derivative $F'(x)$ as $f(x)$, calling $f$ the density function of the random variable $X$. In this case, we can use the fundamental theorem of calculus to write 
\[
\mathbb{P}(X \in (a,b]) = F(b) - F(a) = \int_a^b f(x)~\text{d}x.
\]
Note that by the above rules on $F$, one has $\int_{-\infty}^\infty f(x)~\text{d}x = \lim_{x \to \infty} F(x) = \mathbb{P}(X < \infty) = 1$, and $f$ must be nonnegative. Furthermore we can compute the expectation of $X$ (average value of $X$) as
\[
\mathbb{E}X = \int_{-\infty}^{\infty} x~\text{d}F(x) = \int_{-\infty}^{\infty} x f(x)~\text{d}x.
\]
Generally the expectation of a function $g$ of $X$ is calculated as
\[
\mathbb{E}g(X) = \int_{-\infty}^\infty g(x)~\text{d}F(x) = \int_{-\infty}^\infty g(x)f(x)~\text{d}x.
\]
Some typical random variables are the uniform on $[a,b]$, which has density
\[
f(x) = \begin{cases}
\frac{1}{b-a} & \text{ if } x \in [a,b] \\
0 & \text{ otherwise}
\end{cases}
\]
(in this case we say $X$ has uniform distribution on $[a,b]$) and exponential with parameter $\lambda$, which has density
\[
f(x) = \begin{cases}
\lambda e^{-\lambda x} & \text{ if } x \geq 0 \\
0 & \text{ otherwise}
\end{cases}.
\]
From these density formulas, one can easily calculate, for instance, the ``moments'' of a random variable with the exponential distribution. That is, we can find the $m$-th moment of such an $X$ as
\[
\mathbb{E}X^m = \int_0^\infty x^m f(x) ~\text{d}x = \lambda \int_0^\infty x^m e^{-\lambda x}~\text{d}x,
\]
or the $m$-th central moment
\[
\mathbb{E}|X-\mathbb{E}X|^m = \int_0^\infty |x-\mathbb{E}X|^mf(x)~\text{d}x.
\]
We recall that the second central moment is typically called the variance of $X$:
\[
\mathrm{Var}~X = \mathbb{E}(X-\mathbb{E}X)^2 = \int_0^\infty (x-\mathbb{E}X)^2f(x)~\text{d}x.
\]
In a way, central moments measure how likely $X$ is to vary from its mean, as we have Chebyshev's inequality: for $a > 0$,
\[
\mathbb{P}(|X-\mathbb{E}X| \geq a) \leq \frac{1}{a^m} \mathbb{E}|X-\mathbb{E}X|^m.
\]
For example, since the mean of the exponential distribution is $\lambda^{-1}$, we can compute
\[
\mathbb{P}(|X-\lambda^{-1}| \geq a) = \mathbb{P}(|X-\lambda^{-1}|^2 \geq a^2) \leq \frac{1}{a^2}\mathrm{Var}~X = \frac{1}{(a\lambda)^2}.
\]
As $a \to \infty$, the right side converges to 0, and in a sense, this inequality states that $X$ is reasonably approximated by its mean. 

A fundamental notion in probability is that of independence. Informally, two variables $X$ and $Y$ are independent if given information about $Y$, the distribution of $X$ is unchanged. Recalling that the conditional probability of an event $B$ given $A$ is 
\[
\mathbb{P}(B \mid A) = \frac{\mathbb{P}(A \cap B)}{\mathbb{P}(A)}
\]
(when $\mathbb{P}(A) > 0$), we say that $X,Y$ are independent if whenever $a_i \leq b_i$ for $i=1,2$, one has
\[
\mathbb{P}(X \in [a_1,b_1] \mid Y \in [a_2,b_2]) = \mathbb{P}(X \in [a_1,b_1]),
\]
or
\[
\mathbb{P}(X \in [a_1,b_1], Y \in [a_2,b_2]) = \mathbb{P}(X \in [a_1,b_1]) \mathbb{P}(Y \in [a_2,b_2]),
\]
and the variables $X_1, \ldots, X_n$ are independent if whenever $a_i \leq b_i$ for $i = 1, \ldots, n$, one has
\[
\mathbb{P}\left( \cap_{i=1}^n \left\{X_i \in [a_i,b_i] \right\} \right) = \prod_{i=1}^n \mathbb{P}\left( X_i \in [a_i,b_i] \right).
\]
This condition carries over to expectation, and if $X_1, \ldots, X_n$ are independent, and $g_1, \ldots, g_n$ are integrable functions of the $X_i$'s, then
\[
\mathbb{E}g_1(X_1) \cdots g_n(X_n) = \mathbb{E}g_1(X_1) \cdots \mathbb{E}g_n(X_n).
\]
One use of these factoring properties is to compute the variance of a sum of independent random variables (each of which is assumed to have finite second moment) as
\[
\mathrm{Var}(X_1 + \cdots + X_n) = \mathrm{Var}~X_1 + \cdots + \mathrm{Var}~X_n.
\]
This in turn can be used to control the size of the sum of independent random variables using Chebyshev: if the $X_i$'s all have the same distribution, for $a>0$,
\begin{align*}
\mathbb{P}\left( \left| \frac{X_1 + \cdots + X_n}{n} - \mathbb{E}X_1\right| \geq a\right) &= \mathbb{P}\left( \left| \sum_{i=1}^n X_i - \mathbb{E}\sum_{i=1}^n X_i \right| \geq an \right) \\
&\leq \frac{\mathrm{Var}~\sum_{i=1}^n X_i}{(an)^2} \\
&= \frac{\mathrm{Var}~X_1}{a^2 n},
\end{align*}
which converges to $0$ as $n \to \infty$. This is true for any $a>0$ (however small), so we have proved informally that if $X_1, \ldots, X_n$ are i.i.d. (independent and identically distributed), all with finite second moments, then the Weak 
\index{law of large numbers!WLLN}%
\index{WLLN}%
Law of Large Numbers (WLLN) holds: one has $\frac{X_1 + \cdots + X_n}{n} \to \mathbb{E}X_1$ in probability, or, for any $a>0$,
\[
\mathbb{P}\left( \left| \frac{X_1 + \cdots + X_n}{n} - \mathbb{E}X_1 \right| \geq a\right) \to 0.
\]
In many cases, there is even a rate of convergence on the right side. For instance, we would say that $X_1 + \cdots + X_n - n \mathbb{E}X_1$ is ``exponentially concentrated on scale $\psi(n)$'' if for some $c>0$ and all $t \geq 0$,
\[
\mathbb{P}\left( |X_1 + \cdots + X_n - n \mathbb{E}X_1| \geq t\psi(n) \right) \leq e^{-ct}.
\]
If $\psi(n)$ here can be taken to be of smaller order than $n$, then this is an improved WLLN, as we can take $t = n/\psi(n)$, so that the right side converges to 0 quickly. Exponential concentration is often written in the following form, from which the above follows through a Chebyshev-like inequality: for some $\alpha,C>0$ and all $n$,
\[
\mathbb{E}\exp\left( \alpha \frac{|X_1 + \cdots + X_n - n \mathbb{E}X_1|}{\psi(n)}\right) \leq C<\infty.
\]

There are stronger versions of the above convergence, and the most famous is the Strong 
\index{law of large numbers!SLLN}%
\index{law of large numbers!LLN}%
\index{LLN}%
\index{SLLN}%
Law of Large Numbers (SLLN), which gives an almost sure convergence result: if $(X_i)$ are i.i.d. with finite first moment ($\mathbb{E}|X_1| < \infty$), then one has $\frac{X_1 + \cdots + X_n}{n} \to \mathbb{E}X_1$ almost surely; that is,
\[
\mathbb{P}\left( \frac{X_1 + \cdots + X_n}{n} \to \mathbb{E}X_1 \right) = 1.
\]
Strictly speaking, for this last result to hold, we need to know that we can always find a probability space on which there is an infinite sequence of independent random variables, with specified distributions. This fact is given to us by a technical theorem of Kolmogorov:
\begin{theorem}
Given distribution functions $F_1, F_2, \ldots$, there exists a probability space on which there are independent random variables $X_1, X_2, \ldots$ such that
\begin{enumerate}
\item the $X_i$'s are independent: each finite collection $(X_j)_{j \in J}$ {\rm(}where the set $J \subset \{1, 2, \ldots\}$ is finite{\rm)} of the $X_i$'s is independent, and
\item for each $i$, $X_i$ has distribution function $F_i$.
\end{enumerate}
\end{theorem}
\noindent
This result is used all throughout probability, and we can use it to build our random growth models, like FPP, by generating independent edge-weights $(t_e)$ with specified distributions, one for each edge on an infinite graph, so long as there are countably many edges.

The SLLN tells us that the asymptotic behavior of a sum $S_n = X_1 + \cdots + X_n$ of i.i.d. random variables with finite first moment satisfies almost surely
\[
S_n = n\mathbb{E}X_1 + o(n),
\]
but how large is this error term? Since it is random, the simplest way to quantify it is to find sequences of numbers $(b_n)$ such that
\[
\mathbb{P}\left( S_n - n \mathbb{E}X_1 \geq b_n\right) \to 0 \text{ or } \to 1.
\]
In the case $\to 1$, we would say that $S_n - n \mathbb{E}X_1$ is typically larger than $b_n$, and in the case $\to 0$, smaller than $b_n$. To find the exact scale of $S_n - n \mathbb{E}X_1$, another way is to search for sequences $(b_n)$ such that the convergence above is to a number strictly between $0$ and $1$. Formally, we can look for and a fixed random variable $Z$ (the ``limiting random variable'') and a sequence $(b_n)$ such that for any $t  \in \mathbb{R}$,
\[
\mathbb{P}\left( \frac{S_n - n \mathbb{E}X_1}{b_n} \geq t \right) \to \mathbb{P}(Z \geq t).
\]
In this case, we say that $\frac{S_n-n\mathbb{E}X_1}{b_n}$ converges in distribution to $Z$, write
\[
\frac{S_n - n \mathbb{E}X_1}{b_n} \Rightarrow Z,
\]
and we content ourselves by saying that $S_n - n \mathbb{E}X_1$ ``fluctuates on the scale $b_n$.'' (A technical point here is that we do not need this convergence at all $t$, but just those at which the distribution function of $Z$ is continuous.) This is precisely what the 
\index{central limit theorem}%
Central Limit Theorem (CLT) says:
\begin{theorem}[CLT]
Let $(X_i)$ be i.i.d. such that the $X_i$'s have finite second moment. Then, writing $\mu = \mathbb{E}X_1$ and $\sigma^2 = \mathrm{Var}~X_1$, one has
\[
\frac{S_n - n \mu}{\sqrt{n}\sigma} \Rightarrow Z,
\]
where $Z$ has the standard Gaussian distribution:
\[
\mathbb{P}(Z \geq t) = \frac{1}{\sqrt{2\pi}} \int_t^\infty e^{-\frac{x^2}{2}}~\text{d}x.
\]
\end{theorem}

In the CLT, we see that sums of i.i.d. random variables ``fluctuate on scale $\sqrt{n}$,'' and the limiting distribution is Gaussian. These two statements are quite typical of classical limit theorems in probability. We will see however,
in the rest of the proceedings volume, 
that in random growth models, scales of fluctuation and limiting distributions are not typically of order $\sqrt{n}$ ``diffusive'' and Gaussian, but rather they are of order $n^{1/2-\epsilon}$ for some $\epsilon>0$ ``subdiffusive'' and related to random matrix theory (the Tracy-Widom distribution).

\index{limit shape}%
\section{Limit shapes}

For the rest of the article, we focus on the first two main growth model questions (limit shapes and convergence rate), and only in the context of FPP and LPP. First we will describe results from FPP including shape theorems and properties of limit shapes. Afterward, we switch to LPP to show which results are similar (shape theorems) and which are different (non-polygonal shapes and exactly solvable cases).
\subsection{FPP}

\index{first-passage percolation (FPP)}%
\index{FPP}%
We begin with the shape theorem in FPP. First let's recall the model. We are given the $d$-dimensional cubic lattice $\mathbb{Z}^d$ with the set of its nearest-neighbor edges $\mathcal{E}^d$. We place i.i.d. nonnegative random variables (edge-weights) $(t_e)_{e \in \mathcal{E}^d}$ with common distribution function $F$ on the edges. A path is a sequence $(v_0, e_0, v_1, e_1, \ldots, e_{n-1},v_n)$ of vertices and edges such that $e_i$ has endpoints $v_i$ and $v_{i+1}$. The passage time of a path $\Gamma$ is 
\[
T(\Gamma) = \sum_{e \in \Gamma} t_e
\]
and the passage time between vertices $x,y \in \mathbb{Z}^d$ is
\[
T(x,y) = \inf_{\Gamma : x \to y} T(\Gamma),
\]
where the infimum is over all paths from $x$ to $y$.

\subsubsection{Shape theorem in FPP}
If $\mathbb{P}(t_e=0) = 0$ then $T$ is almost surely a metric on $\mathbb{Z}^d$, as it is nonnegative, one has $T(x,y) = 0$ only when $x=y$, and $T$ satisfies the triangle inequality $T(x,y) \leq T(x,z) + T(z,y)$. (When edge-weights can be zero, $T$ is a pseudometric.) For convenience, we also extend $T$ to real points; that is, for $x \in \mathbb{R}^d$, we set $[x]$ to be the unique point in $\mathbb{Z}^d$ with $x \in [x] + [0,1)^d$. Then the shape theorem is a type law of large numbers for the set
\[
B(t) = \{x \in \mathbb{R}^d : T(0,x) \leq t\},
\]
saying that $B(t)/t$ approaches a limiting set $\mathcal{B}$. 

In the statement of the shape theorem below, $p_c = p_c(d)$ is the threshold for Bernoulli 
\index{bond percolation}%
bond percolation on $\mathbb{Z}^d$ and can be defined as follows. Let $\mathbb{P}_p$ be the probability measure under which edge-weights are Bernoulli with parameter $1-p$:
\[
\mathbb{P}_p(t_e=0) = p \quad \text{and} \quad \mathbb{P}_p(t_e=1)=1-p.
\]
Putting 
\[
\theta(p) = \mathbb{P}_p(0 \text{ is in an infinite self-avoiding path of edges }e \text{ with } t_e=0),
\]
one can then use a coupling argument to show that $\theta$ is nondecreasing in $p$. Therefore the number
\[
p_c = \sup\{p \in [0,1] : \theta(p)=0\}
\]
has the properties $\theta(p) > 0$ for $p > p_c$ and $\theta(p) = 0$ for $p < p_c$. It is an important result that $p_c \in (0,1)$ for all dimensions $d \geq 2$, and we can quickly argue at least that $p_c > 0$. (The argument for $p_c < 1$ can be found in \cite{grimmett}, for example.). Write $A_n$ for the event that there is a self-avoiding path starting from 0 with $n$ edges $e$, all satisfying $t_e=0$. Then for any $n$,
\[
\theta(p) \leq \mathbb{P}_p(A_n) \leq \sum_{\gamma : \#\gamma = n} \mathbb{P}_p(t_e=0 \text{ for all } e \in \gamma),
\]
where the sum is over all self-avoiding paths $\gamma$ starting from $0$ with $n$ edges. Each such $\gamma$ satisfies
\[
\mathbb{P}_p(t_e=0 \text{ for all } e \in \gamma) = p^{\# \gamma},
\]
and since there are $2d(2d-1)^n$ many such paths with $n$ edges (we first choose an edge touching 0, and then at every subsequent step choose an adjacent edge we have not chosen before), we obtain
\[
\theta(p) \leq 2d(2d-1)^n p^n.
\]
This converges to $0$ if $p < 1/(2d-1)$, and so $p_c \geq 1/(2d-1)$.

Returning to the 
\index{shape theorem}%
shape theorem, the condition $\mathbb{P}(t_e=0)<p_c$ below is there to ensure that there are not so many zero-weight edges that $T(0,x)$ will grow sublinearly in $x$. (This occurs for $\mathbb{P}(t_e=0)\geq p_c$.)

\begin{theorem}[Richardson \cite{Richardson}, Cox-Durrett \cite{CoxDurrett}, Kesten \cite{aspects}]
Assume that $\mathbb{E}\min\{t_1, \ldots, t_{2d}\}^d < \infty$, where $t_i$ are i.i.d. copies of $t_e$ and $\mathbb{P}(t_e=0)<p_c$. There exists a deterministic, convex, compact set in $\mathbb{R}^d$, symmetric about the axes and with nonempty interior, such that for any $\epsilon>0$,
\[
\mathbb{P}\left( (1-\epsilon)\mathcal{B} \subset B(t)/t \subset (1+\epsilon)\mathcal{B} \text{ for all large }t \right) = 1.
\]
\end{theorem}
There is also a version for ergodic distributions by Boivin \cite{B90}, first done in less generality by Derriennic, as reported in \cite[(9.25)]{aspects}. One can show that the above is equivalent to: there is a norm $g$ on $\mathbb{R}^d$ such that $\mathcal{B}$ is the unit ball of $g$, and
\[
\limsup_{x \to \infty} \frac{|T(0,x) - g(x)|}{\|x\|_1} = 0 \text{ almost surely}.
\]
This can be thought of as
\[
T(0,x) = g(x) + o(\|x\|_1) \text{ as } x \in \mathbb{R}^d \to \infty.
\]
We will take this approach in the proof: build the norm $g$ and then show the limsup statement above.

\bigskip
\noindent
\begin{proof}[Proof of shape theorem.] The idea of the proof is to first show ``radial'' convergence; that is, for a fixed $x \in \mathbb{Z}^d$, to show that
\[
g(x) := \lim_n \frac{T(0,nx)}{n} \text{ exists}.
\]
To do this, we appeal to the subadditive ergodic theorem. Then we ``patch'' together convergence in many different directions $x$ to get a uniform convergence. We can see quickly that this limit should at least exist when we take expectations, since by invariance under lattice translations,
\[
\mathbb{E}T(x,y) = \mathbb{E}T(x+z,y+z) \text{ for all } x,y,z \in \mathbb{Z}^d,
\]
and so by the triangle inequality, for $0 \leq m \leq n$,
\[
\mathbb{E}T(0,nx) \leq \mathbb{E}T(0,mx) + \mathbb{E}T(mx,nx) = \mathbb{E}T(0,mx) + \mathbb{E}T(0,(n-m)x).
\]
Thus the sequence $(a_n)$ given by $a_n = \mathbb{E}T(0,nx)$ is subadditive and $a_n/n$ must have a limit from the following standard argument (referred to usually as 
\index{Fekete's lemma}%
Fekete's lemma): for fixed $k$ and all $n \geq k$,
\[
a_n \leq a_k + a_{n-k} \leq \cdots \leq \lfloor n/k \rfloor a_k + a_{n-\lfloor n/k \rfloor k}.
\]
Since $n-\lfloor n/k \rfloor k$ is bounded as $n \to \infty$, we divide by $n$ to obtain
\[
\limsup_n a_n/n \leq \limsup_n \frac{\lfloor n/k \rfloor}{n} a_k = a_k/k.
\]
This is true for all $k$, so
\[
\limsup_n a_n/n \leq \inf_k a_k/k \leq \liminf_k a_k/k,
\]
meaning $\lim_n a_n/n$ exists and equals $\inf_k a_k/k$.

To show the limit without expectation is much harder, and we will have to appeal to some machinery that was invented specifically for this problem. Liggett's version \cite[Theorem~1.10]{LiggettST} of Kingman's 
\index{subadditive ergodic theorem}%
subadditive ergodic theorem states:
\begin{theorem}\label{thm: subadd}
Let $\{X_{m,n} : 0 \leq m < n\}$ is an array of random variables satisfying the following assumptions:
\begin{enumerate}
\item for each $n$, $\mathbb{E}|X_{0,n}| <\infty$ and $\mathbb{E}X_{0,n} \geq -cn$ for some constant $c>0$,
\item $X_{0,n} \leq X_{0,m} + X_{m,n}$ for $0 < m < n$,
\item for each $m \geq 0$, the sequence $\{X_{m+1,m+k+1} :  k \geq 1\}$ is equal in distribution to the sequence $\{X_{m,m+k} : k \geq 1\}$, and
\item for each $k \geq 1$, $\{X_{nk, (n+1)k} : n \geq 1\}$ is a stationary ergodic process.
\end{enumerate}
Then 
\[
g : = \lim_n \frac{1}{n} \mathbb{E}X_{0,n} = \inf_n \frac{1}{n} \mathbb{E}X_{0,n} \text{ exists},
\]
and
\[
\lim_n \frac{1}{n} X_{0,n} =g \text{ a.s. and in }L^1.
\]
\end{theorem}
\begin{proof}
We will give a proof of a version of the subadditive ergodic theorem in the next section.
\end{proof}

We apply this theorem for a fixed $x \in \mathbb{Z}^d$ to the sequence
\[
X_{m,n} = T(mx,nx).
\]
Item 2 holds by the triangle inequality, whereas 3 and 4 hold by invariance of the environment under integer translations. In item 1 we can take any $c>0$, since $T \geq 0$ a.s. The only thing to check is that $\mathbb{E}T(0,nx) < \infty$ for each $n$. By subadditivity, and symmetry, it suffices to check that $\mathbb{E}T(0,e_1) < \infty$. To do this, we construct $2d$ edge-disjoint deterministic paths $\gamma_1, \ldots, \gamma_{2d}$ from 0 to $e_1$ and note
\[
\mathbb{E}T(0,e_1) \leq \mathbb{E}\min\{T(\gamma_1), \ldots, T(\gamma_{2d})\}.
\]
Now we can check the following result from \cite{CoxDurrett}:

\begin{lemma}
If $\mathbb{E}\min\{t_1, \ldots, t_{2d}\}<\infty$ for i.i.d. edge-weights $t_i$, then the right side above is finite.
\end{lemma}
\begin{proof}
Let $k$ be the maximal number of edges of any $\gamma_i$ and note that for each $i$,
\[
\mathbb{P}\left( T(\gamma_i) \geq \lambda \right) \leq k \mathbb{P}(t_e \geq \lambda/k),
\]
so that if we put $M = \min\{T(\gamma_1), \ldots, T(\gamma_{2d})\}$,
\begin{align*}
\mathbb{P}(M \geq \lambda) = \prod_{i=1}^{2d} \mathbb{P}(T(\gamma_i) \geq \lambda) &\leq k^{2d} \left( \mathbb{P}(t_e \geq \lambda/k) \right)^{2d} \\
&= k^{2d} \mathbb{P}\left( \min\{t_1, \ldots, t_{2d}\} \geq \lambda/k \right).
\end{align*}
Therefore
\[
\mathbb{E}M = \int_0^\infty \mathbb{P}(M \geq \lambda)~\text{d}\lambda \leq k^{2d} \int_0^\infty \mathbb{P}\left( \min\{t_1, \ldots, t_{2d}\} \geq \lambda/k \right)~\text{d}\lambda < \infty.\qedhere
\]
\end{proof}

By the subadditive ergodic theorem, then, we define
\[
g(x) = \lim_n \frac{T(0,nx)}{n},
\]
which is a.s. constant. We next extend $g$ to $\mathbb{Q}^d$ by taking such a rational $x$ and letting $m \in \mathbb{N}$ be such that $mx$ is an integer point. Then set
\[
g(x) = \lim_n \frac{T(0,mnx)}{mn} = \frac{1}{m} \lim_n \frac{T(0,n(mx))}{n} = \frac{1}{m} g(mx).
\]
$g$ thus defined on $\mathbb{Q}^d$ satisfies the following properties: for $x,y \in \mathbb{Q}^d$,
\begin{enumerate}
\item $g(x+y) \leq g(x) + g(y)$,
\item $g$ is uniformly continuous on bounded sets,
\item for $q \in \mathbb{Q}$, $g(qx) = |q|g(x)$,
\item $g$ is symmetric about the axes.
\end{enumerate}
Item 1 follows from the triangle inequality for $T$, and 4 follows from symmetries of the edge weights. Item 3 is an easy exercise, and 2 follows from 1: for $h=(h_1, \ldots, h_d)$,
\begin{align*}
|g(z) - g(z+h)| &\leq \max\{g(h),g(-h)\} = g(h) = g(h_1e_1 + \cdots + h_de_d)\\ 
&\leq g(e_1)(|h_1| + \cdots + |h_d|) \leq \|h\|_1 \mathbb{E}T(0,e_1).
\end{align*}
Then $g$ has a continuous extension to $\mathbb{R}^d$. The above properties extend to real arguments, so $g$ is a seminorm. (A norm, except it could have $g(x) = 0$ for some $x \neq 0$.) It is a result of Kesten \cite[Theorem~6.1]{aspects} that $g$ is a norm when $\mathbb{P}(t_e=0)<p_c$.

Now that we have ``radial'' convergence to a norm $g$, we need to patch together convergence in every direction to a type of uniform convergence. Here we do this under the unnecessary but simplifying assumption:
\[
\mathbb{P}(t_e \in [a,b]) = 1, \text{ where } 0 < a < b < \infty.
\]
This implies
\[
T(0,x), g(x) \in [a\|x\|_1, b\|x\|_1] \text{ for } x \in \mathbb{Z}^d.
\]
So define the event
\[
\Omega' = \{t_e \in [a,b] \text{ for all }e \} \cap \left\{ \lim_n \frac{T(0,nx)}{n} = g(x) \text{ for all } x \in \mathbb{Z}^d\right\}.
\]
Note that on the above event we actually have
\begin{equation}\label{eq: bellegrande}
\lim_{\alpha \to \infty} \frac{T(0,\alpha x)}{\alpha} = g(x) \text{ for all } x \in \mathbb{Z}^d.
\end{equation}
(Here $\alpha$ is real instead of just being an integer.) The reason is that we can bound
\begin{align*}
&\left| \frac{T(0,\alpha x)}{\alpha} - g(x)\right|\\ 
&\qquad\leq \left| \frac{T(0,\lfloor \alpha \rfloor x)}{\lfloor \alpha \rfloor} - \frac{T(0,\alpha x)}{\lfloor \alpha \rfloor} \right| + \left| \frac{T(0,\alpha x)}{\lfloor \alpha \rfloor} - \frac{T(0,\alpha x)}{\alpha} \right| + \left| \frac{T(0,\lfloor \alpha \rfloor x)}{\lfloor \alpha \rfloor} - g(x) \right| \\
&\qquad\leq \frac{T(\lfloor \alpha \rfloor x,\alpha x)}{\lfloor \alpha \rfloor} + T(0,\alpha x)\left| \frac{1}{\lfloor \alpha\rfloor} - \frac{1}{\alpha} \right|+ \left| \frac{T(0,\lfloor \alpha \rfloor x)}{\lfloor \alpha \rfloor} - g(x)\right|.
\end{align*}
By the assumed inequality $t_e \leq b$, the first two terms go to zero as $\alpha \to \infty$. The last term approaches zero because of condition of the second intersected event in the definition of $\Omega'$.

Fix $\omega \in \Omega'$, a set which has probability one, for the rest of the argument. We will show the equivalent statement
\[
\limsup_{x \to \infty} \frac{|T(0,x) - g(x)|}{\|x\|_1} = 0,
\]
so suppose it fails: for some $\epsilon>0$, there is a sequence $(x_n)$ going to infinity such that
\[
|T(0,x_n) - g(x_n)| \geq \epsilon \|x_n\|_1 \text{ for all } n.
\]
We may assume by compactness that $x_n/\|x_n\|_1 \to z$ for some $z$ with $\|z\|_1=1$. We will show that for some $x$ in a nearby direction to $z$, one cannot have $T(0,nx) / n \to g(x)$. (See Figure~\ref{fig: fig_1} for an illustration.) 

\begin{figure}[!ht]
\centering
\includegraphics[width=4in,trim={0 17cm 10cm 2cm},clip]{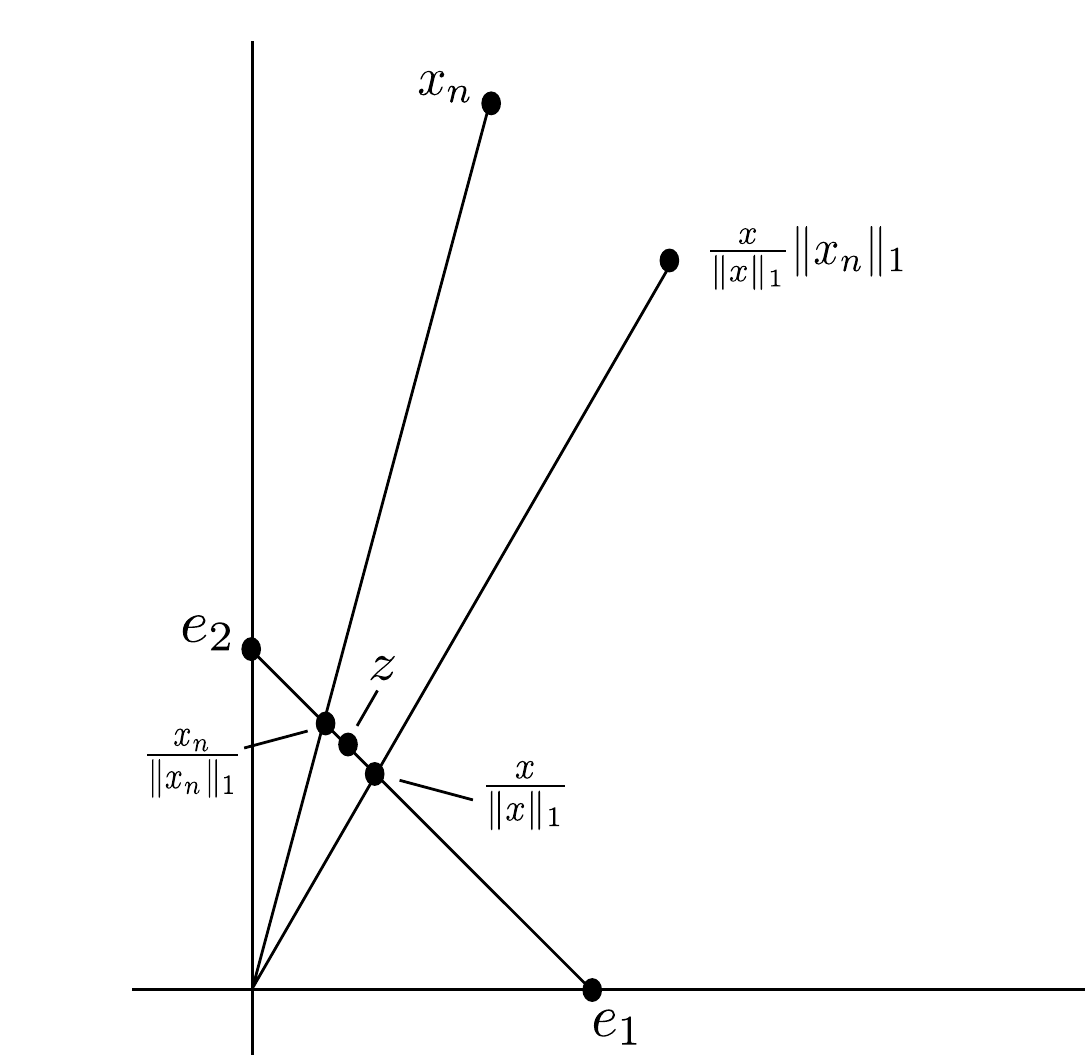}
\caption{Illustration of the argument given in the proof of the shape theorem. The points $\frac{x_n}{\|x_n\|_1}$, $z$, and $\frac{x}{\|x\|_1}$ are all of $\ell^1$-norm 1 and within $\ell^1$-distance $2\delta$ of each other. The point $x$ is chosen in $\mathbb{Z}^d$ so that we have radial convergence of $T(0,nx)/n$ to $g(x)$ in this direction. Because $x_n$ is close in direction to $x$, and our weights are bounded between $a$ and $b$, the passage time $T(0,x_n)$ cannot be too different (on a linear scale) than passage time from $0$ to the ``nearby point'' $\frac{x}{\|x\|_1}\|x_n\|_1$. The latter passage time is controlled, since it is from 0 to a multiple of $x$, so this prohibits the sequence $T(0,x_n)$ from fluctuating too much.}
\label{fig: fig_1}
\end{figure}

\noindent
Fix
\[
\delta \in \left( 0 , \frac{\epsilon}{4b+1} \right)
\]
and choose $x \in \mathbb{Z}^d$ such that
\[
\left\| \frac{x}{\|x\|_1} - z \right\|_1 < \delta,
\]
so that for $n$ large,
\[
\left\| \frac{x_n}{\|x_n\|_1} - \frac{x}{\|x\|_1} \right\|_1 < 2\delta.
\]
We will compare the passage time from $0$ to $x_n$ to the passage time to the ``nearby'' point $\|x_n\|_1 x/\|x\|_1$. Then
\begin{align*}
|T(0,x_n) - g(x_n)| 
&\leq \left|T(0,x_n) - T\left(0, \|x_n\|_1 \frac{x}{\|x\|_1} \right)\right|\\
&\qquad + \left| T\left( 0, \|x_n\|_1 \frac{x}{\|x\|_1} \right) - g\left( \|x_n\|_1 \frac{x}{\|x\|_1} \right) \right| \\
&\qquad + \left| g\left( \|x_n\|_1 \frac{x}{\|x\|_1} \right) - g(x_n) \right|.
\end{align*}
The first and last terms are bounded by
\[
b\left\| \|x_n\|_1 \frac{x}{\|x\|_1} - x_n \right\|_1 = b\|x_n\|_1 \left\| \frac{x}{\|x\|_1} - \frac{x_n}{\|x_n\|_1} \right\|_1 < 2b \delta \|x_n\|_1.
\]
However since we have radial convergence in direction $x$, \eqref{eq: bellegrande} gives
\[
T\left( 0, \|x_n\|_1 \frac{x}{\|x\|_1} \right) = g\left( \|x_n\|_1 \frac{x}{\|x\|_1} \right) + o(\|x_n\|_1),
\]
so for $n$ large, the second term is bounded by $\delta \|x_n\|_1$. In total,
\[
\epsilon \|x_n\|_1 \leq |T(0,x_n) - g(x_n)| \leq (4b+1)\delta \|x_n\|_1 < \epsilon \|x_n\|_1,
\]
a contradiction.
\end{proof}


\index{subadditive ergodic theorem}%
\subsubsection{Proof of subadditive ergodic theorem}
In this section, we prove a version of the subadditive ergodic theorem using ideas from \cite{DH2} (which turn out to be similar to those of Liggett's proof). The current version developed through my discussions with F.\ Rassoul-Agha, and it is his notes we give here.

We prove the following version of the theorem, which is weaker than that stated in the last section, but sufficient for our purposes in FPP.
\begin{theorem}
Let $(X_{m,n} : m,n \in \mathbb{Z})$ be an array of random variables satisfying the following assumptions.
\begin{enumerate}
\item For all $m,n \in \mathbb{Z}$, $X_{n,n} = 0$ and $X_{n,m} = X_{m,n}$,
\item for each $n \in \mathbb{Z}$, $\mathbb{E}|X_{0,n}| < \infty$,
\item for all $m,n,l \in \mathbb{Z}$, $X_{l,n} \leq X_{l,m} + X_{m,n}$,
\item for each $m \in \mathbb{Z}$, the sequence $(X_{m+1,m+k+1} : k \in \mathbb{Z})$ is equal in distribution to the sequence $(X_{m,m+k} : k \in \mathbb{Z})$, and
\item for each $k \geq 1$, the process $(X_{nk,(n+1)k} : n \in \mathbb{Z})$ is stationary and ergodic.
\end{enumerate}
Then
\begin{equation}\label{eq: to_show_erg_1}
g := \lim_{n \to \infty} \frac{\mathbb{E}X_{0,n}}{n} = \inf_{n\geq 1} \frac{\mathbb{E}X_{0,n}}{n} \text{ exists}
\end{equation}
and
\[
\lim_{n \to \infty} \frac{X_{0,n}}{n} = g \text{ a.s. and in }L^1.
\]
\end{theorem}
\begin{proof}
Note first that \eqref{eq: to_show_erg_1} holds by subadditivity of the sequence $\mathbb{E}X_{0,n}$ (which itself follows from stationarity in item (4)), and 
\index{Fekete's lemma}%
Fekete's lemma (mentioned above Theorem~\ref{thm: subadd}). Also observe that (1) and (3) imply that $0 = X_{n,n} \leq X_{n,m} + X_{m,n} = 2 X_{m,n}$; that is, the random variables are nonnegative.

For $k \geq 1$, let $\mu_k$ be the distribution of 
\[
(X_{m,n}, X_{r,k} - X_{l,k})_{m,n,r,l\in \mathbb{Z}}.
\]
This is a probability measure on the space of sequences 
\[
\widetilde \Omega = \{ (x_{m,n}, b_{r,l})_{m,n,r,l \in \mathbb{Z}}\},
\]
equipped with the product topology and its Borel sigma-field. Define the measure $\nu_L$ as
\[
\nu_L = \frac{1}{L} \sum_{k=1}^L \mu_k.
\]
By subadditivity, we have
\[
X_{m,k} \leq X_{m,n} + X_{n,k} \quad \text{and} \quad X_{n,k} \leq X_{n,m}+X_{m,k} = X_{m,n} + X_{m,k},
\]
so
\begin{equation}\label{eq: erg_eq_1}
|X_{m,k} - X_{n,k}| \leq X_{m,n} \text{ for all }k.
\end{equation}

{\bf Exercise.} Prove that $(\nu_L : L \geq 1)$ is a tight family of probability measures.\medskip

By the above exercise and Prohorov's theorem (see, for instance \cite[Theorem~3.2.7]{DurrettBook}), there is a subsequence $(L_j)$ such that $\nu_{L_j}$ converges weakly to a limiting probability measure $\nu$ on $\widetilde \Omega$. This measure has the property that its marginal on the sequences $(x_{m,n} : m,n \in \mathbb{Z})$ is precisely the distribution of $(X_{m,n} : m,n \in \mathbb{Z})$. Indeed, for bounded continuous $f = f(x_{m,n} : m,n \in \mathbb{Z})$, we have
\begin{align*}
\mathbb{E}_\nu f(X_{m,n} : m,n \in \mathbb{Z}) &= \lim_{j \to \infty} \frac{1}{L_j} \sum_{k=1}^{L_j} \mathbb{E}_{\mu_k} f(X_{m,n} : m,n \in \mathbb{Z}) \\
&= \lim_{j \to \infty} \frac{1}{L_j} \sum_{k=1}^{L_j} \mathbb{E}f(X_{m,n} : m,n \in \mathbb{Z}) \\
&= \mathbb{E}f(X_{m,n} : m,n \in \mathbb{Z}).
\end{align*}

The probability measure $\nu$ has the following properties (writing as above $X_{m,n}$, $B_{m,n}$ as the coordinate random variables on $\widetilde \Omega$)
\begin{enumerate}
\item[(a)] $|B_{m,n}| \leq X_{m,n}$, $\nu$-a.s.,
\item[(b)] $B_{l,n} = B_{l,m} + B_{m,n}$, $\nu$-a.s., and for all $l,m,n \in \mathbb{Z}$,
\item[(c)] for any $N \geq 1$, $(B_{nN,(n+1)N} : n \in \mathbb{Z})$ is stationary, and
\item[(d)] $\mathbb{E}_\nu B_{0,1} = g$.
\end{enumerate}
We will now (mostly) justify the above four points. We will use the following fact:\medskip

{\bf Exercise.} Suppose that $(A_n,B_n)$ converge in distribution to $(A,B)$. Assume also that $A_n \leq B_n$ a.s. Show that $A \leq B$ a.s. Prove the same result for $\leq$ replaced by $=$.\medskip

Property (a) comes from \eqref{eq: erg_eq_1} and applying the above exercise. Property (b) comes by another application of the exercise and telescoping the sum
\[
X_{l,k} - X_{n,k} = X_{l,k} - X_{m,k} + X_{m,k} - X_{n,k}.
\]
Property (c) is a consequence of stationarity (item 4) in the statement of the theorem, and it is why averaging was used to define $\nu_L$. More precisely, if $f = f(b_{nk,(n+1)k} : n \in \mathbb{Z})$ is bounded and continuous, then
\begin{align*}
&\mathbb{E}_\nu f(B_{nN,(n+1)N} : n \in \mathbb{Z}) = \lim_{j \to \infty} \frac{1}{L_j} \sum_{k=1}^{L_j} \mathbb{E}_{\mu_k} f(B_{nN,(n+1)N} : n \in \mathbb{Z}) \\
&\qquad\qquad= \lim_{j \to \infty} \frac{1}{L_j} \sum_{k=1}^{L_j} \mathbb{E}f(X_{nN,k}-X_{(n+1)N,k} : n \in \mathbb{Z}) \\
&\qquad\qquad= \lim_{j \to \infty} \frac{1}{L_j} \sum_{k=1-N}^{L_j-N} \mathbb{E}f(X_{nN,k}-X_{(n+1)N,k} : n \in \mathbb{Z}) \\
&\qquad\qquad= \lim_{j \to \infty} \frac{1}{L_j} \sum_{k=1-N}^{L_j-N} \mathbb{E}f(X_{(n+1)N,k+N}-X_{(n+2)N,k+N} : n \in \mathbb{Z}) \\
&\qquad\qquad= \lim_{j \to \infty} \frac{1}{L_j} \sum_{k=1}^{L_j} \mathbb{E}f(X_{(n+1)N,k} - X_{(n+2)N,k} : n \in \mathbb{Z}) \\
&\qquad\qquad= \lim_{j \to \infty} \frac{1}{L_j} \sum_{k=1}^{L_j} \mathbb{E}_{\mu_k} f(B_{(n+1)N,(n+2)N} : n \in \mathbb{Z}) \\
&\qquad\qquad= \mathbb{E}_\nu f(B_{(n+1)N,(n+2)N} : n \in \mathbb{Z}).
\end{align*}
The third equality follows from the fact that $f$ is bounded, and so changing finitely many terms in the sum does not affect the limit. The fourth equality is from item 4 in the statement of the theorem.

To show item (d), we will give the idea, and the rest is an exercise. If we knew that the coordinates $B_{m,n}$ were bounded, then weak convergence would give us
\begin{align*}
\mathbb{E}_\nu B_{0,1} = \lim_{j \to \infty} \frac{1}{L_j} \sum_{k=1}^{L_j} \mathbb{E}_{\mu_k} B_{0,1} &= \lim_{j \to \infty} \frac{1}{L_j} \sum_{k=1}^{L_j} \mathbb{E}(X_{0,k} - X_{1,k}) \\
&= \lim_{j \to \infty} \frac{1}{L_j} \left( \sum_{k=1}^{L_j} \mathbb{E}X_{0,k} - \sum_{k=1}^{L_j} \mathbb{E}X_{1,k} \right) \\
&= \lim_{j \to \infty} \frac{1}{L_j} \left( \sum_{k=1}^{L_j} \mathbb{E}X_{0,k} - \sum_{k=1}^{L_j} \mathbb{E}X_{0,k-1} \right) \\
&= \lim_{j \to \infty} \frac{1}{L_j} \mathbb{E}X_{0,L_j} \\
&= g.
\end{align*}
(The last limit holds by \eqref{eq: to_show_erg_1}.) However, $f$ is not bounded.\medskip

{\bf Exercise.} Fix the above argument by truncating and using \eqref{eq: erg_eq_1}.\medskip

Now, abbreviate $m_n = \lfloor n/N \rfloor$. Subadditivity implies that
\[
X_{0,n} \leq X_{0,N} + X_{N,2N} + \cdots + X_{(m_n-1)N,m_nN} + X_{m_nN,nN}.
\]

{\bf Exercise.} Prove that if $f$ is an $L^1$ function defined on a probability space whose measure is invariant under a measurable map $\theta$, then $f \circ \theta^n / n$ converges to 0 a.s.\medskip

Since the process $(X_{nN, (n+1)N} : n \in \mathbb{Z})$ is assumed to be ergodic, by the above exercise and the ergodic theorem,
\begin{align*}
\limsup_{n \to \infty} \frac{1}{n} X_{0,n} &\leq \lim_{n \to \infty} \frac{m_n}{n} \frac{X_{0,N} +\cdots + X_{(m_n-1)N,m_nN}}{m_n} \\
&+ \lim_{n \to \infty} \frac{m_n}{n} \frac{\max_{1 \leq l \leq N} X_{0,l} \circ \theta^{m_n N}}{m_n} \\
&= \frac{1}{N} \mathbb{E}X_{0,N}.
\end{align*}
Here, the map $\theta$ acts on a sequence $a=(a_n)_{n \in \mathbb{Z}}$ by shifting to the left: $(\theta a)_n = a_{n+1}$. The maximum comes because $n -N \leq m_n N \leq n$, and so $n-m_n N \in \{1, \ldots, N\}$. Taking infimum over all $N$ and applying \eqref{eq: to_show_erg_1}, we have
\begin{equation}\label{eq: erg_eq_2}
\limsup_{n \to \infty} \frac{1}{n} X_{0,n} \leq g.
\end{equation}
Then by property (a) of the process $(B_{m,n})$, we have $\nu$-a.s.
\[
\lim_{n \to \infty} \frac{B_{0,n}}{n} \leq \limsup_{n \to \infty} \frac{X_{0,n}}{n} \leq g.
\]
Since $(B_{n,n+1})$ is stationary, additivity (item (b)) and the ergodic theorem imply that $\nu$-a.s.,
\[
\lim_{n \to \infty} \frac{B_{0,n}}{n} = \lim_{n \to \infty} \frac{1}{n} \sum_{k=0}^{n-1} B_{k,k+1} = \mathbb{E}_\nu[B_{0,1} \mid \mathcal{I}],
\]
where $\mathcal{I}$ is the invariant sigma-field. Putting these two displays together, we obtain
\[
\mathbb{E}_\nu[B_{0,1} \mid \mathcal{I}] \leq g.
\]
However, property (d) gives
\[
\mathbb{E}_\nu \left[ \mathbb{E}_\nu[B_{0,1} \mid \mathcal{I}]\right] = \mathbb{E}_\nu B_{0,1} = g,
\]
and so
\[
\lim_{n \to \infty} \frac{1}{n}B_{0,n} = g,~\nu\text{-a.s.}.
\]

Apply item (a) again, combined with \eqref{eq: erg_eq_2}, to obtain
\[
g = \lim_{n \to \infty} \frac{1}{n} B_{0,n} \leq \liminf_{n \to \infty} \frac{1}{n} X_{0,n} \leq \limsup_{n \to \infty} \frac{1}{n} X_{0,n} \leq g.
\]
This means that $\lim_{n \to \infty} \frac{1}{n} X_{0,n} = g$ $\nu$-a.s. Since this is the same as $\mathbb{P}$-a.s., we are done.
\end{proof}

\begin{remark}
If we knew the existence of a.s. limits $B_{m,n} = \lim_{k \to \infty} (X_{m,k} - X_{n,k})$, then the proof would be more direct. In other words, there would be no need for the measures $\mu_k$ and $\nu_L$, and one could prove properties (a)-(d) and deduce the result similarly to the above.
\end{remark}

\begin{remark}
In the version of the theorem in the previous section, one only takes an array $(X_{m,n})$ with $n > m$. In this case, one needs an extra assumption which in a sense replaces nonnegativity, and it appears as the second part of condition (1): $\mathbb{E}X_{0,n} \geq -cn$ for a constant $c>0$. Then one needs an additional argument for tightness in the proof.
\end{remark}

\subsubsection{High dimensions}\label{sec: high_dimensions_FPP}

In FPP, the shape theorem gives that $\mathcal{B}$ is convex, compact, with nonempty interior, and with all the symmetries of $\mathbb{Z}^d$. Many different sets have this property, in particular all the $\ell^p$ balls for $p \in [1,\infty]$. So the shape theorem allows in principle a cube ($\ell^\infty$ ball) and a diamond ($\ell^1$ ball), and therefore says nothing about strict convexity, flat edges, corners, or whether the shape is a polygon. 

For $(t_e)$ that are i.i.d. and, say, continuous, the following properties are expected for $\mathcal{B}$.
\begin{enumerate}
\item $\mathcal{B}$ is 
\index{strict convexity}%
strictly convex. That is, for any distinct $x,y \in \mathcal{B}$ and $\lambda \in (0,1)$ the point $\lambda x + (1-\lambda)y$ is in the interior of $\mathcal{B}$. Thus $\mathcal{B}$ should have no flat facets (like a polygon has).
\item $\mathcal{B}$ has no ``corners.'' That is, the boundary of $\mathcal{B}$ should be differentiable. One way to say this is in terms of supporting hyperplanes. A hyperplane in $\mathbb{R}^d$ is a set of the form $\{x = (x_1, \ldots, x_d) : x \cdot y = a\}$ for some $y \in \mathbb{R}^d$ and $a \in \mathbb{R}$ (where `$\cdot$' is the standard dot product $x \cdot y = \sum_{i=1}^d x_iy_i$). A hyperplane $H$ is supporting for $\mathcal{B}$ at $z \in \partial \mathcal{B}$ if $H$ contains $z$ but $\mathcal{B}$ intersects at most one component of $H^c$. By the Hahn-Banach Theorem, since $\mathcal{B}$ is convex and bounded, each $z \in \partial \mathcal{B}$ has a supporting hyperplane. We say that $\partial \mathcal{B}$ is differentiable if each $z \in \partial \mathcal{B}$ has exactly one supporting hyperplane. (This rules out ``corners'' for $\mathcal{B}$, as in the case of a polygon.)
\item Combining the above two cases, but weaker, the set $\mathcal{B}$ should not be a polygon. To state this precisely, we say that $\mathcal{B}$ should have infinitely many extreme points. An extreme point $x \in \mathcal{B}$ is a point which is not on the interior of a line segment with endpoints in $\mathcal{B}$. Precisely, whenever we write $x = \lambda z + (1-\lambda)y$ with $z,y \in \mathcal{B}$ and $\lambda \in (0,1)$, we have $z=y$. So $\mathcal{B}$ has infinitely many extreme points.
\item The boundary $\partial \mathcal{B}$ should have uniformly positive 
\index{curvature}%
curvature. In other words, near every boundary point, the boundary should locally look like the boundary of a Euclidean ball with bounded radius.
\end{enumerate}

Here we state the main high-dimensional asymptotic result. It says that for distributions with no atom at 0 but with 0 in the support (like exponential), for high dimensions, the limit shape is not a Euclidean ball, an $\ell^1$-ball, or an $\ell^\infty$-ball. Since we saw that the Eden model is equivalent to FPP with exponential weights, this shows the same result for the Eden model in high dimensions. For the statement, we define the balls $\mathbf{D}, \mathbf{B},\mathbf{C}$ to be the $\ell^1,\ell^2,$ and $\ell^\infty$-balls of radius $g(e_1)^{-1}$. We will also make the following assumptions from \cite{AT}:
\begin{equation}\label{eq: moment_AT}
\mathbb{E}t_e<\infty,
\end{equation}
and for some $a \geq 0$,
\begin{equation}\label{eq: others_AT}
F(0) = 0 \text{ and } \left|\frac{\mathbb{P}(t_e \leq x)}{x} - a \right| = O(|\log x|^{-1}) \text{ as } x \downarrow 0.
\end{equation}

The following theorem is from Auffinger-Tang \cite{AT}, which weakens various assumptions (widens the class of distributions in particular) of the work of previous authors. Some earlier work was done by Kesten \cite{aspects}, Dhar \cite{Dhar}, Couronn\'e-Enriquez-Gerin \cite{CEG11}, and Martinsson \cite{Martinsson}.
\begin{theorem}
Assume \eqref{eq: moment_AT} and \eqref{eq: others_AT}. For all large $d$, in FPP on $\mathbb{Z}^d$ with fixed weights $(t_e)$, one has
\[
\mathbf{D} \subset \mathcal{B} \subset \mathbf{C} \text{ and } \mathcal{B} \neq \mathbf{B}, \mathbf{C}, \text{ or } \mathbf{D}.
\]
\end{theorem}

\begin{itemize}
\item The proof proceeds by showing the asymptotic
\[
\lim_{d \to \infty} \frac{g(e_1)d}{\log d} = \frac{1}{2a}.
\]
For example, if $\mathcal{B}$ were the $\ell^1$-ball, one would have for $e = (1, \ldots, 1)$,
\[
g(e) = dg(1/d, \ldots, 1/d) = dg(e_1) \sim \frac{\log d}{2a}.
\]
However one can show that $g(e) \leq C$ for some constant $C$. In fact, Combining the work of Martinsson \cite{Martinsson} and Couronn\'e-Enriquez-Gerin \cite{CEG11}, one has can show that $\lim_{d \to \infty} g(e)$ exists for the exponential distribution and is related to the nonzero solution of $\coth \alpha = \alpha$.
\item It is not hard to give the bound for exponential weights $g(e) \leq 1$. Construct a path $\gamma$ with vertices $x_0, x_1, \ldots$ as follows: set $x_0=0$ and for $n \geq 0$, starting from $x_n$, take the minimal weight edge of the $d$ different edges leading in directions $e_1, \ldots, e_d$ (the positive coordinate directions) and call its endpoint $x_{n+1}$. We see that if $X_1, X_2, \ldots$ are the weights of the first edge, second edge, and so on, then $X_i$ is the minimum of $d$ i.i.d. exponential random variables, so is an exponential with mean $1/d$. Thus putting $H_n = \{x = (x_1, \ldots, x_d) : \sum_i x_i = n\}$ one obtains
\[
\mathbb{E}T(0,H_n) \leq \sum_{i=1}^n \mathbb{E}X_i = n/d.
\]
But one can show using the shape theorem that
\[
g(e) = \lim_n \frac{\mathbb{E}T(0,ne)}{n} = \lim_n \frac{\mathbb{E}T(0,H_{dn})}{n} \leq \lim_n \frac{\sum_{i=1}^{nd} \mathbb{E}X_i}{n} = 1
\]
\end{itemize}

\index{flat edges}%
\subsubsection{Flat edges}\label{sec: flat_edge}

A main question in FPP is: which compact convex sets are realizable as limit shapes? The question above is completely open in the i.i.d. setting. Interestingly, though, this is solved by H\"aggstr\"om and Meester '95 \cite{HM} in the case of stationary (not necessarily i.i.d.) passage times. 
\begin{theorem}[H\'aggstr\"om-Meester]\label{th:Hag-Mee} 
Any non-empty compact, convex set $C$ that has the symmetries of $\mathbb{Z}^d$ is a limit shape for some FPP model with weights distributed according to a stationary (under translations of $\mathbb Z^d$) and ergodic measure.
\end{theorem}
From this result, we see for example that there are stationary-ergodic models of FPP that have limit shapes equal to polygons or Euclidean balls. So certainly conjectured properties from the i.i.d. setting like strict convexity are not true in such generality. 

The previous section focused on high dimensions. What can we say about low dimensions? For one special class of distributions, Durrett and Liggett \cite{DurrettLiggett} were able to say much more: that the limit shape is not 
\index{strict convexity}%
strictly convex. Of course we believe strict convexity in the continuous weight case, so these distributions have atoms. Their ``flat edge'' result holds for higher dimensions as well, but we can give a precise description of it in two dimensions, and further work has been done by Marchand \cite{Marchand}, Zhang \cite{Zhangsuper, Zhang}, and Auffinger-Damron \cite{AD12}.

Recalling that $F$ is the distribution of our weights, let $\mathcal{M}_p$ be the set of distributions $F$ that satisfy the following:
\begin{enumerate}
\item $F(x)=0$ for all $x < 1$ and $F(1) = p \geq \vec p_c$, where $\vec p_c$ is the 
\index{oriented percolation}%
oriented bond percolation threshold (approximately .70548), and
\item $\int x ~\text{d}F(x) = \mathbb{E}t_e < \infty$.
\end{enumerate}

In \cite{DurrettLiggett}, it was shown that if $F \in \mathcal{M}_p$ then the limit shape $\mathcal B$ has some flat edges. The precise location of these edges was found in \cite{Marchand}. To describe this, write $\mathbf{B}_1$ for the closed $\ell^1$ unit ball:
\[
\mathbf{B}_1 = \{(x,y)\in \mathbb{R}^2~:~ |x|+|y|\leq 1\}
\]
and write $int ~\mathbf{B}_1$ for its interior. For $p > \vec p_c$ let $\alpha_p$ be the asymptotic speed of oriented percolation (see \cite{Durrett}), define the points
\begin{equation}\label{eq:NP}
M_p = \left(\frac{1}{2} - \frac{\alpha_p}{\sqrt 2}, \frac{1}{2} + \frac{\alpha_p}{\sqrt 2}\right) \text{ and } N_p = \left(\frac{1}{2} + \frac{\alpha_p}{\sqrt 2}, \frac{1}{2} - \frac{\alpha_p}{\sqrt 2}\right)
\end{equation}
and let $[M_p,N_p]$ be the line segment in $\mathbb{R}^2$ with endpoints $M_p$ and $N_p$. For symmetry reasons, the following theorem is stated only for the first quadrant.

\begin{theorem}[Durrett-Liggett \cite{DurrettLiggett},  Marchand \cite{Marchand}]\label{thm:marchand1}
Let $F \in \mathcal{M}_p$ in two dimensions.
\begin{enumerate}
\item $\mathcal B \subset \mathbf{B}_1$.
\item If $p < \vec p_c$ then $\mathcal B \subset int~\mathbf{B}_1$.
\item If $p > \vec p_c$ then $\mathcal B \cap [0,\infty)^2 \cap \partial \mathbf{B}_1 = [M_p,N_p]$.
\item If $p = \vec p_c$ then $\mathcal B \cap [0,\infty)^2 \cap \partial \mathbf{B}_1 = (1/2,1/2)$.
\end{enumerate}
\end{theorem}

The angles corresponding to points in the line segment $[M_p,N_p]$ are said to be in the 
\index{percolation cone}%
\emph{percolation cone}. 
The existence of a flat edge for the limit shape can be explained heuristically. The definition of the oriented percolation threshold $\vec p_c$ is as follows, taking $\mathbb{P}_p$ to be the distribution of i.i.d. edge-weights $(\eta_e)$ on $\mathcal{E}^2$ with probability $p$ to be 1 and $1-p$ to be 2:
\[
\vec p_c = \sup\{p : \mathbb{P}_p(\exists ~\text{infinite oriented path }\Gamma \text{ with } \eta_e = 1 \text{ for } e \in \Gamma) = 0\}.
\]
(Here, oriented means as usual that the vertices of the path, in order, have non-decreasing coordinates.) By monotonicity of the probability in $p$, one has that for $p>\vec p_c$, if $F \in \mathcal{M}_p$, then there is positive probability of existence of an infinite oriented path of all $1$-weights starting from 0. In fact one can even show the stronger statement that for any angle $\theta$ in the percolation cone, there is positive probability for an infinite oriented path of $1$-weights starting from 0 and going in direction $\theta$ (the angles of the vertices on the path converge to $\theta$). Since no edge-weights have value below 1, any finite segment of such an infinite path must be a minimizing path (geodesic). But distance along these geodesics correspond to the $\ell^1$-distance on $\mathbb{Z}^2$, so in such directions, the limit shape must correspond to the $\ell^1$-ball. It is important to point out that it is unknown (but expected to be false) whether there are other distributions whose limit shapes have flat edges. Even in the case of $\mathcal{M}_p$, the flat part of the percolation cone ends at $M_p$ and $N_p$; however, that does not exclude the limit shape from having further flat spots. This is not expected though.


Let $\beta_p := 1/2 + \alpha_p/\sqrt{2}$, that is, define $\beta_p$ as the $x$ coordinate of $N_p$. Convexity and symmetry of the limit shape imply that $1/g(e_1) \geq \beta_p$. A non-trivial statement about the edge of the percolation cone came in 2002 when Marchand \cite[Theorem 1.4]{Marchand} proved that this inequality is in fact strict:
\[
1/g(e_1) > 1/2 + \alpha_p/\sqrt{2}.
\] 
In other words, Marchand's result says that the line that goes through $N_p$ and is orthogonal to the $e_1$-axis is not a tangent line of $\partial \mathcal B$. The following theorem builds on Marchand's result and technique and says that at the edge of the percolation cone, one cannot have a corner.

\begin{theorem}[Auffinger-Damron \cite{AD12}]\label{thm: diffll}
Let $F \in \mathcal{M}_p$ for $p\in[\vec p_c,1)$. The boundary $\partial \mathcal B$ is differentiable at $M_p$ and $N_p$.
\end{theorem}


The theorem above shows that any distribution in $\mathcal M_p$ has a 
\index{non-polygonal limit shape}%
non-polygonal limit shape. 
The first example of a non-polygonal limit shape was discovered by Damron-Hochman \cite{DH}.




\subsection{LPP}

\index{last-passage percolation (LPP)}%
\index{LPP}%
LPP is defined similarly to FPP, but taking the maximum over oriented paths, instead of the infimum over all paths. On $\mathbb{Z}^d$, we assign i.i.d. site-weights $(t_v)_{v \in \mathbb{Z}^d}$ with common distribution $F$. These weights no longer need to be nonnegative, but for simplicity we will take them to be so. An oriented path $\Gamma$ with vertices $x_0, x_1, \ldots, x_n$ has the property that all coordinates of $x_i$ are less or equal to the corresponding coordinates of $x_{i+1}$ (this is written $x_i \leq x_{i+1}$); that is, for all standard basis vectors $e_k$, and all $i$, one has $x_i \cdot e_k \leq x_{i+1} \cdot e_k$. By convention, we identify the path $\Gamma$ with its vertices, but we exclude the initial point. The passage time of such a path is $T(\Gamma) = \sum_{v \in \Gamma} t_v$, and the last-passage time $T(x,y)$ between vertices $x \leq y$ is
\[
T(x,y) = \max_{\Gamma : x \to y} T(\Gamma),
\]
where the maximum is over oriented paths from $x$ to $y$. Note that there are only finitely many paths under consideration, so we can take a maximum. It is important to note that when $x,y$ do not satisfy $x \leq y$ or $y \leq x$, then $T(x,y)$ is not defined. Also, by convention, $T(x,x) = 0$ for all $x$.

Due to directedness of the model and the fact that we are taking a maximum, $T$ has somewhat different properties than those in FPP. One still has for $x \leq y$, $T(x,y) \geq 0$ and if $t_v > 0$ for all $v$, then $T(x,y) > 0$ when $x \neq y$. Due to excluding the initial point from all our paths, we have a super-additivity property of $T$ that corresponds to the triangle inequality in FPP:
\[
\text{for } x \leq y \leq z,~ T(x,z) \geq T(x,y) + T(y,z).
\]
Due to this super-additivity, the limiting shape in LPP is not convex, since the corresponding ``shape function'' $g$ will be super-additive. For its definition, we again appeal to the subadditive ergodic theorem, noting that $-T$ is subadditive. The only difficulty is to come up with conditions under which the limit is finite. The following version of ``radial convergence'' for $T$ to a limit comes from Martin \cite{Martin}.
\begin{theorem}
Suppose that 
\begin{equation}\label{eq: shape_condition}
\int_0^\infty (1-F(x))^{1/d}~\text{d}x < \infty.
\end{equation}
Then for each $x \geq 0$ in $\mathbb{R}^d$, the following (deterministic) limit exists a.s. and in $L^1$:
\[
g(x) = \lim_n \frac{T(0,[nx])}{n} < \infty,
\]
where $[nx]$ is the point of $\mathbb{Z}^d$ with $nx \in [nx]+[0,1)^d$. The function $g$ is continuous on $\{x : x \geq 0\}$ (including at the boundaries), satisfies $g(x+y) \geq g(x) + g(y)$, is invariant under permutations of the coordinates, and $g(ax) = ag(x)$ for $a \geq 0$.
\end{theorem}

Although $\mathbb{E}t_v < \infty$ is sufficient to guarantee that $\mathbb{E}T(0,x) < \infty$ for all $x \geq 0$ (just bound $T(0,x)$ above by the sum of all weights $t_y$ with $0 \leq y \leq x$), the subadditive ergodic theorem this time gives
\[
g(x) = \sup_{n \geq 1} \frac{\mathbb{E}T(0,nx)}{n},
\]
and to deduce that $g(x)<\infty$, we need to know that the means $\mathbb{E}T(0,nx)$ do not grow faster than linearly. The corresponding issue in FPP is that the means $\mathbb{E}T(0,nx)$ do not grow sublinearly, which is guaranteed by $F(0) < p_c$. Here the corresponding condition is not on $F(0)$ but, as we are in LPP, it is on the tail of $F$, and is the integrability condition \eqref{eq: shape_condition}.

As in FPP, one can upgrade radial convergence to a sort of 
\index{shape theorem}%
shape theorem \cite{Martin}. As before, put
\[
\mathcal{B} = \{x \geq 0 : g(x) \leq 1\}.
\]
\begin{theorem}[LPP shape theorem]
Assume \eqref{eq: shape_condition} and put $B(t) = \{x \geq 0 : T(0,x) \leq t\}$. For any $\epsilon>0$, one has
\[
\mathbb{P}\left( (1-\epsilon)\mathcal{B} \subset B(t)/t \subset (1+\epsilon)\mathcal{B} \text{ for all large }t \right) = 1.
\]
\end{theorem}

\subsubsection{Near the boundary}
\index{non-polygonal limit shape}%
Just as in FPP, not much is known about the limiting shape $\mathcal{B}$. It is expected as before to have differentiable boundary (at least when the weights are continuous) with positive curvature, and certainly not to be a polygon. The flat edge result from FPP carries over to LPP: the analogous condition on our distribution $F$ is
\[
F(1) =1,~ F(1^-) = 1-p, \text{ where } p > \vec p_c.
\]
Such distributions $F$ of weights $t_v$ have $\mathbb{P}(t_v>1)=0$ but $\mathbb{P}(t_v = 1)=p>\vec p_c$. The value $c$ is the highest weight of an edge, and any oriented path of all $c$-weights will be an optimal infection path. Thus again in the percolation cone, the boundary of the limit shape will agree with that of the $\ell^1$-ball.

Somewhat surprisingly, in two dimensions, the limit shape is shown not to be a polygon for most distributions. This is in contrast to the situation in FPP, where this is only known for distributions in $\mathcal{M}_p$. The LPP result is a corollary of a ``universality'' of the shape function $g$ near the boundary. The following result, proved by Martin \cite{Martinboundary}, shows that the asymptotics of $g$ near the boundary of the quarter plane $\{x : x \geq 0\}$ are explicit, and only depend on the mean and variance of $F$. From these asymptotics we can extract non-polygonality of the limit shape.
\begin{theorem}
Consider $d=2$. Write $\mu$ for the mean of $F$ and $\sigma^2$ for the variance of $F$. If $F$ satisfies \eqref{eq: shape_condition}, then
\[
g(1,a) = \mu + 2\sigma \sqrt{a} + o(\sqrt{a}) \text{ as } a \downarrow 0.
\]
\end{theorem}
Here, $g(1,a)$ is the function $g$ evaluated at the point $(1,a)$, and as $a \downarrow 0$, this point approaches the boundary of the quarter plane. Note that $g(1,0) = \mu$, since the passage time from $0$ to $ne_1$ must be achieved along the one oriented path connecting these points, and its passage time is the sum of $n+1$ i.i.d. random variables with distribution $F$. Thus the law of large numbers gives the value of $g$ at $(1,0)$. Therefore the above result says:
\[
g(e_1 + ae_2) - g(e_1) = 2\sigma \sqrt{a} + o(\sqrt{a}) \text{ as } a \downarrow 0.
\]

To see that the above result implies that the limit shape is not a polygon, suppose for a contradiction that the limit shape is a polygon. Then it must have finitely many extreme points, and the boundary of the shape between the extreme points consists of straight line segments. If the limit shape is a triangle, then by symmetry, $g$ must be a multiple of the $\ell^1$-norm, but then $g(e_1 + ae_2) - g(e_1) = a$ for $a > 0$ and the above asymptotics are violated. Otherwise, there is a closest extreme point $w$ to the point $(1,a)$, and the limit shape boundary must be a line segment between $w$ and $(1,a)$. In this case, $g(e_1 + ae_2) - g(e_1) = ca$ for some real $c$ and $a$ small enough. But this again violates the above asymptotics.

\subsubsection{Exactly solvable cases in two dimensions}

The most famous case of LPP is when the distribution $F$ of the site-weights is exponential in two dimensions. Here, there is a direct mapping from the growth of $B(t)$ to a particle system called the Totally Asymmetric Simple Exclusion Process 
\index{TASEP}%
(TASEP). TASEP is defined loosely as follows. We imagine that at each site $z$ of $\mathbb{Z}$ with $z \leq 0$, there sits a particle at time 0. Associated to each particle is a Poisson process, and when the process increments, the particle attempts to move to the site directly to the right. If there is already a particle there, the move is suppressed, and the particle stays in its current location. The particle that is initially at site 0 is allowed to move unrestricted (since there are never any particles to the right of it), but the other particles may sometimes be blocked by particles to their rights. Our convention is that the particle at 0 immediately moves to the right at time zero. That is, at time zero, there is a particle at site 1, and particles at sites $-k$ for $k \geq 1$.

What is the relation between TASEP and LPP with exponential weights? To begin, the procession of the first particle in TASEP is the same as the infection in LPP along the positive $e_1$-axis from 0. Indeed, the infection appears at site 0 at time 0, just as the first particle in TASEP moves to the right. It then infects $e_1$ after an independent exponential time, just as the same particle in TASEP moves again to the right. Generally, the infection time from 0 to $ne_1$ is achieved through the path that proceeds directly down the positive $e_1$-axis, and occurs when the first particle in TASEP reaches site $n+1$. 

\begin{figure}[!ht]
\centering
\includegraphics[width=4.5in,trim={4cm 19cm 8cm 7cm},clip]{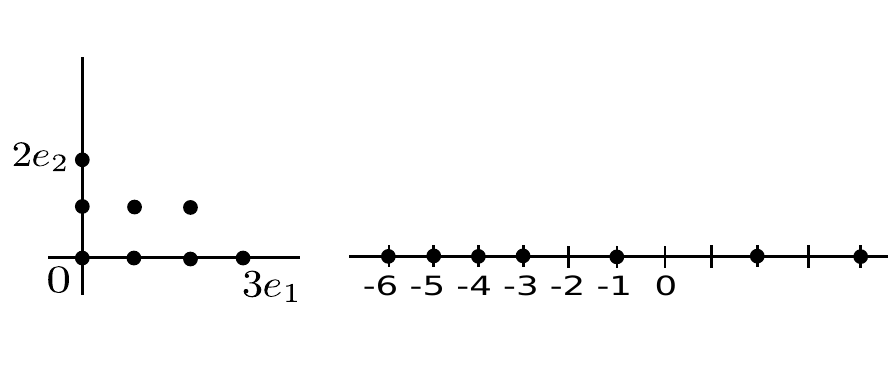}
\caption{Illustration of the correspondence between LPP and TASEP. In the left (LPP), the infection has moved 4 steps to the right on the $e_1$-axis, and so on the right (TASEP) first particle has moved four steps to the right, from 0 to 4. Similarly, the infection has taken three steps at the second level, and the second particle has moved three steps to the right from $-1$ to $2$. Last, the third TASEP particle has taken one step from $-2$ to $-1$.}
\label{fig: fig_2}
\end{figure}

At the second level, the infection of site $e_2$ occurs an independent exponential time after the infection appears at 0. This corresponds to the second particle in TASEP moving into the space left open after the first particle moves. Generally, the $n$-th step of the $k$-th particle in TASEP corresponds to the site $ke_2 + (n-1)e_1$ being infected from $0$. To see this, we can derive the following relation in LPP: for $x_1,x_2>0$, one has
\[
T(0,(x_1,x_2)) = t_{(x_1,x_2)} + \max\left\{ T(0,(x_1-1,x_2)), T(0,x_1,x_2-1)\right\}.
\]
This is because the infection from 0 reaches $(x_1,x_2)$ through either $(x_1-1,x_2)$ or $(x_1,x_2-1)$ (whichever is infected last), and after the one of these sites with maximal passage time from 0 is infected, $(x_1,x_2)$ must wait $t_{(x_1,x_2)}$ additional time. Similarly, in TASEP, for the $k$-th particle to make its $n$-th step, it must wait an independent exponential time after both of the following events occur: (a) the $k-1$-st particle makes its $n$-th step and (b) the $k$-th particle makes its $n-1$-st step. See Figure~\ref{fig: fig_2} for an example.

Because of this coupling, we can represent the passage times in LPP in terms of ``currents'' in TASEP. For example, 
\[
T(0,(n,n)) = \text{ time for }n\text{-th particle to make }n+1\text{-st step}.
\]
But this is exactly the time needed for $n$ particles to pass through the origin. Thus if we define $c_t$ as the ``current through zero at time $n$''; that is, the number of particles having passed through 0 by time $t$, one has
\[
T(0,(n,n)) \leq t \Leftrightarrow c_t \geq n.
\]
In a pioneering work, Rost \cite{Rost} showed in '81 exact asymptotics for variables like this current, and this translates directly to an exact formula for the limit shape in exponential LPP.
\begin{theorem}
Let $(t_v)$ be exponentially distributed with mean 1. Then 
\[
g(x) = g(x_1,x_2) = (\sqrt{x_1} + \sqrt{x_2})^2.
\]
\end{theorem}
Therefore the limit shape boundary in the above case is
\[
\left\{ (x,y) \in \mathbb{R}^2 : x,y \geq 0 \text{ and } \sqrt{x} + \sqrt{y} = 1 \right\}.
\]
So we can see directly that the limit shape is not a polygon, contains no flat segments, and has no ``corners.''

In the case of geometrically distributed weights with parameter $p$ \cite{CEP,JPS,S98}, there is an exact formula as well, showing
\[
g(x_1,x_2) = \frac{1}{p} (x_1 + x_2 + 2 \sqrt{x_1x_2(1-p)}).
\]
In both of these cases, finer asymptotics are available \cite{Joha} by relating the law of the passage time to the largest eigenvalue of a random matrix ensemble. Specifically for $x = (x_1,x_2)$ with $x_1>0$ and $x_2>0$,
\[
\frac{T(0,nx) - ng(x)}{n^{1/3}} \Rightarrow Z \text{ as } n \to \infty
\]
for a nondegenerate variable $Z$. This shows that the fluctuations of $T$ about $g$ are of order $n^{1/3}$, and this is expected for general distributions in both LPP and FPP in two dimensions. Unfortunately, despite decades of work, the best available bounds for $|T(0,nx) - ng(x)|$ are of order $\sqrt{n}$.

\index{rate of convergence}%
\section{Rate of convergence and scaling exponents}

In this section, we restrict to the case of FPP for simplicity, although similar theorems are provable in LPP.

\subsection{Decomposition of error and strengthened shape theorems in FPP}

The question we address here is: is it possible to improve the shape theorem to one with $\epsilon$ that depends on $t$? In other words, can we find $\epsilon_t$ such that $\epsilon_t \downarrow 0$ quickly and
\[
\mathbb{P}\left( (1-\epsilon_t)\mathcal{B} \subset B(t)/t \subset (1+\epsilon_t)\mathcal{B} \text{ for all large } t \right) = 1?
\]
This is a question about the rate of convergence in the shape theorem. Restating it in terms of the norm $g$, we can ask the question in the following form: writing $T(0,x) = g(x) + o(\|x\|)$, how large is the term $o(\|x\|)$? The standard way to study this is to decompose
\[
o(\|x\|) = \left[T(0,x) - \mathbb{E}T(0,x)\right] ~+~ \left[\mathbb{E}T(0,x) - g(x)\right]
\]
into a random fluctuation term and a 
\index{nonrandom fluctuation}%
nonrandom fluctuation term. 

The random fluctuation term is typically analyzed using techniques from concentration of measure -- there many techniques (transportation inequalities, entropy methods, influence inequalities, isoperimetry, etc.) have been developed to study deviations of a function $f(X_1, X_2, \ldots)$ of independent variables $(X_n)$ away from its mean or its median. Despite all these methods, however, current bounds on the first term are far from the predictions 
(and we will see this in the articles on fluctuations).   

The nonrandom term is purely deterministic, and can be written as 
\[
\|x\| \left[ \frac{\mathbb{E}T(0,x)}{\|x\|} - g\left( \frac{x}{\|x\|} \right) \right] \geq 0.
\]
The term in the parenthesis is of the form $\mathbb{E}T(0,nx)/n - g(x)$ for $x$ on the unit circle. We have seen that for $x \in \mathbb{Z}^d$, the sequence $(\mathbb{E}T(0,nx))$ is subadditive and so, when divided by $n$, converges. So quantifying this nonrandom error is really a problem of estimating the rate of convergence of $a_n/n$ to its limit for a subadditive sequence $(a_n)$. Unfortunately there are no general methods for this, but in the context of lattice models (like FPP), techniques have been developed to bound these nonrandom fluctuations in terms of the random ones. Specifically, if one has a concentration inequality of the form
\[
\mathbb{P}\left( T(0,x) - \mathbb{E}T(0,x) < - \lambda \|x\|^\alpha \right) \leq \exp\left( - C \lambda^{\beta}\right),~ \lambda \geq 0
\]
for constants $C,\alpha,\beta > 0$, then one can show an upper bound of the type
\[
\mathbb{E}T(0,x) - g(x) \leq C \|x\|^\alpha (\log \|x\|)^\delta,
\]
implying that nonrandom fluctuations should be no larger than random fluctuations. Indeed, using Gaussian concentration inequalities, one has the following version of Alexander's \cite{Alexander} result from Damron-Kubota \cite{DKubota}.
\begin{theorem}\label{thm: taco_taco}
Assume that $\mathbb{E}\min\{t_1, \ldots, t_d\}^2 < \infty$ and $\mathbb{P}(t_e=0)<p_c$. Then for some $C>0$ one has
\[
g(x) \leq \mathbb{E}T(0,x) \leq g(x) + C \sqrt{\|x\| \log \|x\|} \text{ for all } \|x\| > 1.
\]
\end{theorem}

On the other hand, evidence from \cite{ADH} implies that random fluctuations should be no larger than nonrandom fluctuations. Therefore if we posit the 
\index{scaling exponents}%
existence of exponents such that
\[
T(0,x) - \mathbb{E}T(0,x) \sim \|x\|^{\chi} \text{ and } \mathbb{E}T(0,x) - g(x) \sim \|x\|^\gamma
\]
then one should have $\chi = \gamma$. (Note that Theorem~\ref{thm: taco_taco} is a version of $\gamma \leq 1/2$. Also, since the first term is random, it may be measured as $\mathrm{Var}~T(0,x) \sim \|x\|^{2\chi}$ or in terms of a concentration inequality.) Unfortunately this is far from the state of art, as under various general assumptions (exponential moments for the passage times, for instance), the best existing bounds are
\[
0 \leq \chi \leq 1/2 \text{ and } -1/2 \leq \gamma \leq 1/2.
\]
(Under strong assumptions on existence of a fluctuation exponent $\chi < 1/2$, one can show $\chi = \gamma$ \cite{ADH}.) As we saw in the section on exactly solvable models of LPP, we believe that in two dimensions, $\chi = \gamma = 1/3$. It is reasonable to expect that these (equal) numbers decrease strictly with dimension, and approach 0 as dimension tends to infinity. Much more on the random fluctuation term and concentration estimates will be given in article \cite{ch:Sosoe} 

In summary, we have a strengthened shape theorem of the following type. Some improvements to the assumptions have been made by Tessera \cite{Tessera} and Damron-Kubota \cite{DKubota} more recently, in particular generally replacing the log with $\sqrt{\log}$. The term $\sqrt{t}$ comes from the bounds $\chi, \gamma \leq 1/2$.
\begin{theorem}[Rate of convergence bound in the shape theorem]
Assume that $\mathbb{E}e^{\alpha t_e} < \infty$ for some $\alpha > 0$. There is $C>0$ such that 
\[
\mathbb{P}\left( (t-C\sqrt{t \log t}) \mathcal{B} \subset B(t) \subset (t+C\sqrt{t \log t}) \text{ for all large } t \right) = 1.
\]
\end{theorem}
It is reasonable to believe, as we saw above, that a stronger shape theorem may hold, with $\sqrt{t}$ replaced by $t^{\chi} = t^{\gamma}$ for the fluctuation exponents explained above.

\index{scaling exponents}%
\subsection{Scaling exponents and the 
\index{Kardar-Parisi-Zhang (KPZ)!scaling relation}%
KPZ relation}\label{sec:scale+KPZ}

There are no accepted definitions of the exponents $\chi$ and $\gamma$ from the last section. One can define, as in \cite{ADH}, directional $p$-fluctuation exponents ($p \geq 1$) in direction $x \in \mathbb{Z}^d$ as
\[
\underline{\chi}_p(x) = \liminf_{n \to \infty} \frac{\log \|T(0,nx) - \mathbb{E}T(0,nx)\|_p}{\log n}
\]
and
\[
\overline{\chi}_p(x) = \limsup_{n \to \infty} \frac{\log \|T(0,nx)- \mathbb{E}T(0,nx)\|_p}{\log n},
\]
where $\|X\|_p = \left(\mathbb{E}|X|^p\right)^{1/p}$, and 
\index{nonrandom fluctuation}%
nonrandom fluctuation exponents
\[
\underline{\gamma}(x) = \liminf_{n \to \infty} \frac{\log (\mathbb{E}T(0,nx) - g(nx))}{\log n}
\]
and
\[
\overline{\gamma}(x) = \limsup_{n \to \infty} \frac{\log (\mathbb{E}T(0,nx) - g(nx))}{\log n}.
\]
At the moment, it is not known if $\underline{\chi}_p(x) = \overline{\chi}_p(x)$ for any $p$ or $x$ generally, or if $\underline{\gamma}(x) = \overline{\gamma}(x)$. We can precisely state the bounds from the last section in terms of these exponents: it is known, combining work of Alexander \cite{Alexander}, Kesten \cite{kestenspeed}, and Auffinger-Damron-Hanson \cite{ADH} that under an exponential moment assumption (this can be weakened in various cases) that
\[
0 \leq \underline{\chi}_p(x) \leq \overline{\chi}_p(x) \leq 1/2 \text{ for all } p \geq 1 \text{ and } x \neq 0
\]
and
\[
-1 \leq \underline{\gamma}(x) \text{ and } -1/2 \leq \overline{\gamma}(x) \leq 1/2 \text{ for all } x \neq 0.
\]
Our lack of information on the model leads to other (stronger) direction-independent definitions. We will present these below while discussing the so-called KPZ scaling relation.

The last exponent we need is the ``wandering exponent,'' which measures the maximal displacement of the (random) geodesics from Euclidean straight lines. Following Chatterjee \cite{Sutav}, for any nonzero $x$, let $D(0,x)$ be the maximal Euclidean distance between any point on a geodesic (minimizing path for $T$) from $0$ to $x$ and the line segment connecting $0$ and $x$. It is reasonable to believe that
\[
\mathbb{E}D(0,x) \sim \|x\|^\xi
\]
for some (dimension-dependent) $\xi = \xi(d)$. The predictions are that in two dimensions, $\xi = 2/3$, that $\xi$ decreases with $d$ to 1/2, but might always be $>1/2$. In fact, these statements can be read directly off of the similar predictions for $\chi$, along with the conjectured 
\index{Kardar-Parisi-Zhang (KPZ)!scaling relation}%
KPZ scaling relation
\[
\chi = 2\xi - 1.
\]
This relation is expected to be ``universal.'' That is, it does not depend on the distribution of the $(t_e)$'s, and does not even depend on the dimension $d$, as long as the distribution of the $(t_e)$'s is reasonable, say with no atoms, and enough moments. There are heuristic arguments from physics for this relation in \cite{Krug}, but to date, there is no proof that is valid in full generality; that is, there is no proof under only edge-weight assumptions. One major difficulty is that there is no accepted definition of exponents which is usable. For instance, the exponents $\overline{\chi}_p$ and $\underline{\chi}_p$ defined above always exist, but unless one can show they are equal, they are not so helpful. Furthermore, the current bounds on the exponent $\xi$ depend even on which definition is taken! In Newman-Piza \cite{NP} and Licea-Newman-Piza \cite{LNP}, the only work giving general bounds on $\xi$, one has versions of
\[
\xi \geq 1/(d+1), ~ \xi \geq 1/2,~ \xi \geq 3/5,~ \text{ and } \xi \leq 3/4,
\]
depending on the definition of $\xi$. For example, the first bound is valid for a quite general definition of $\xi$, the second for a point-to-line geodesic wandering exponent, the third for a more restricted point-to-line exponent, and the fourth only in directions of ``positive curvature'' of the limiting shape. So, for instance, it is not even known at this point if
\[
\mathbb{E}D(0,ne_1) = o(n).
\]

Regardless of the precise definition, Newman-Piza provided the first rigorous argument for the inequality $\chi \geq 2\xi-1$, and it is essentially from this inequality and the known bound $\chi \leq 1/2$ that they derive $\xi \leq 3/4$. It took another 16 years for a rigorous argument, due to Chatterjee, for the other inequality, $\chi \leq 2\xi-1$, even under strong assumptions of existence of exponents. A couple of months later, a simplified proof due to Auffinger-Damron \cite{AD} appeared, which removed a technical assumption on the valid class of distributions. We stress though that these inequalities are still conditional, in that they assume existence of exponents $\chi$ and $\xi$.

Below we will give a proof sketch for the KPZ relation. We begin with Chatterjee's exponents. 
\index{scaling exponents}%
\begin{definition}
$\chi_u$ is the smallest number such that for any $\chi' > \chi_u$, there exists $\alpha>0$ such that
\[
\sup_{x \neq 0} \mathbb{E}\exp\left( \alpha \frac{|T(0,x) - \mathbb{E}T(0,x)|}{\|x\|^{\chi'}} \right) < \infty,
\]
and $\chi_\ell$ is the largest number such that for any $\chi'' < \chi_{\ell}$, one has
\[
\inf_{x \neq 0} \frac{\mathrm{Var}~T(0,x)}{\|x\|^{\chi''}} > 0.
\]
$\xi_u$ is the smallest number such that for any $\xi' > \xi_u$, there exists $\beta>0$ such that
\[
\sup_{x \neq 0} \mathbb{E}\exp\left( \beta \frac{D(0,x)}{\|x\|^{\xi'}} \right) < \infty
\]
and $\xi_\ell$ is the largest number such that for any $\xi'' < \xi_\ell$,
\[
\inf_{x \neq 0} \frac{\mathbb{E}D(0,x)}{\|x\|^{\xi''}} > 0.
\]
\end{definition}
There are two important things to notice about these definitions. First, they are not directional, as they all take supremum or infimum over all directions $x$. So, for instance, if we assume that $\chi_u = \chi_\ell$ (as we will in the theorem below), then we are assuming that random fluctuations are the same in all directions, which rules out the case of the class $\mathcal{M}_p$ from Section~\ref{sec: flat_edge} (see \cite{Zhang} for more details). For continuous distributions, though, this should be true. Next, the upper exponents are somewhat stronger than we might want, as they incorporate information about exponential concentration of the variables $T(0,x)$ and $D(0,x)$. In other words, these definitions imply (by an application of Markov's inequality for the exponential): for $\chi' > \chi_u$, there exists $C_1,C_2>0$ such that for all $x$ 
\[
\mathbb{P}\left( |T(0,x) - \mathbb{E}T(0,x)| > \|x\|^{\chi'} \right) \leq C_1\exp\left( - \|x\|^{C_2} \right),
\]
and for $\xi' > \xi_u$, there are $C_3,C_4$ such that for all $x$,
\begin{equation}\label{eq: geo_concentration}
\mathbb{P}\left( D(0,x) \geq \|x\|^{\xi'} \right) \leq C_3\exp\left( - \|x\|^{C_4} \right).
\end{equation}
In fact, by using Alexander's technique (as shown in \cite{Sutav}), the first inequality can be upgraded to
\begin{equation}\label{eq: time_concentration}
\mathbb{P}\left( |T(0,x) - g(x)| > \|x\|^{\chi'} \right) \leq C_1 \exp\left( - \|x\|^{C_2} \right).
\end{equation}

Now we state the KPZ relation that has been rigorously proved to date, combining the results of Chatterjee \cite{Sutav} and Auffinger-Damron \cite{AD}.
\index{Kardar-Parisi-Zhang (KPZ)!scaling relation}%
\begin{theorem}[Scaling relation]\label{thm: KPZ}
Assume that $\mathbb{P}(t_e=0) < p_c$. Then if $\chi := \chi_\ell = \chi_u$ and $\xi:= \xi_\ell = \xi_u$, one has $\chi = 2\xi-1$.
\end{theorem}

Before the sketch, we mention some comments. The above relation has been extended to a positive temperature model, directed polymers in a random environment \cite{AD13}. Next, curvature exponents $\kappa$ for the boundary of the limit shape have been defined in \cite{AD}, and it is believed that $\kappa = 2$. This corresponds to a boundary which is locally like Euclidean ball. For other curvature exponents, arguments from \cite{AD} suggest that a different relation holds: $\chi = \kappa \xi - (\kappa - 1)$. Last, versions of this relation have been proved in two dimensions in the exactly solvable version of LPP and other continuum models, where it is known that $\chi = 1/3$ and $\xi = 2/3$. In fact, the argument of \cite{AD} for the upper bound on $\chi$ is quite similar to the one by W\"uthrich \cite{Wuthrich} and later by Johansson \cite{Joha1}.

\begin{proof}[Sketch of proof of Theorem~\ref{thm: KPZ}]
We will give the proofs in the $e_1$ direction for simplicity. For the lower bound, suppose for a contradiction that $\xi > \frac{1+\chi}{2}$ and choose $\xi' \in \left( \frac{1+\chi}{2}, \xi \right)$. We will show that
\begin{equation}\label{eq: to_show_bound}
\mathbb{P}(D(0,ne_1) > n^{\xi'}) \leq e^{-n^c}
\end{equation}
for some $c>0$. This itself suggests a contradiction, since $\xi' < \xi$. A more careful argument would actually give $\mathbb{E} D(0,ne_1) = O(n^{\xi'})$, which contradicts $\xi' < \xi = \xi_\ell$.

To show \eqref{eq: to_show_bound}, put $L$ to be the line segment connecting $0$ and $ne_1$, and consider the set 
\[
S = \{x \in \mathbb{Z}^d : \|x - u\| < n^{\xi'} \text{ for some }u \in L\},
\]
with
\[
\partial S = \{y \in \mathbb{Z}^d \setminus S : \|y-x\|_1 = 1 \text{ for some } x \in S\}.
\]
If $D(0,ne_1) > n^{\xi'}$, then there is $z \in \partial S$ such that $z$ is on a geodesic from $0$ to $ne_1$, and so $A_z =\{T(0,ne_1) = T(0,z) + T(z,ne_1)\}$ occurs. Therefore
\[
\mathbb{P}(D(0,ne_1) > n^{\xi'}) \leq \sum_{z \in \partial S} \mathbb{P}(A_z).
\]
One can now show that there is $c>0$ depending on $\chi'$ such that
\begin{equation}\label{eq: wandering_bound}
\mathbb{P}(A_z) \leq e^{-n^c} \text{ for all } z \in \partial S.
\end{equation}
Summing over all $z$, we then obtain \eqref{eq: to_show_bound}.

\begin{figure}[!ht]
\centering
\includegraphics[width=4in,trim={.5cm 18cm 6.5cm 6cm},clip]{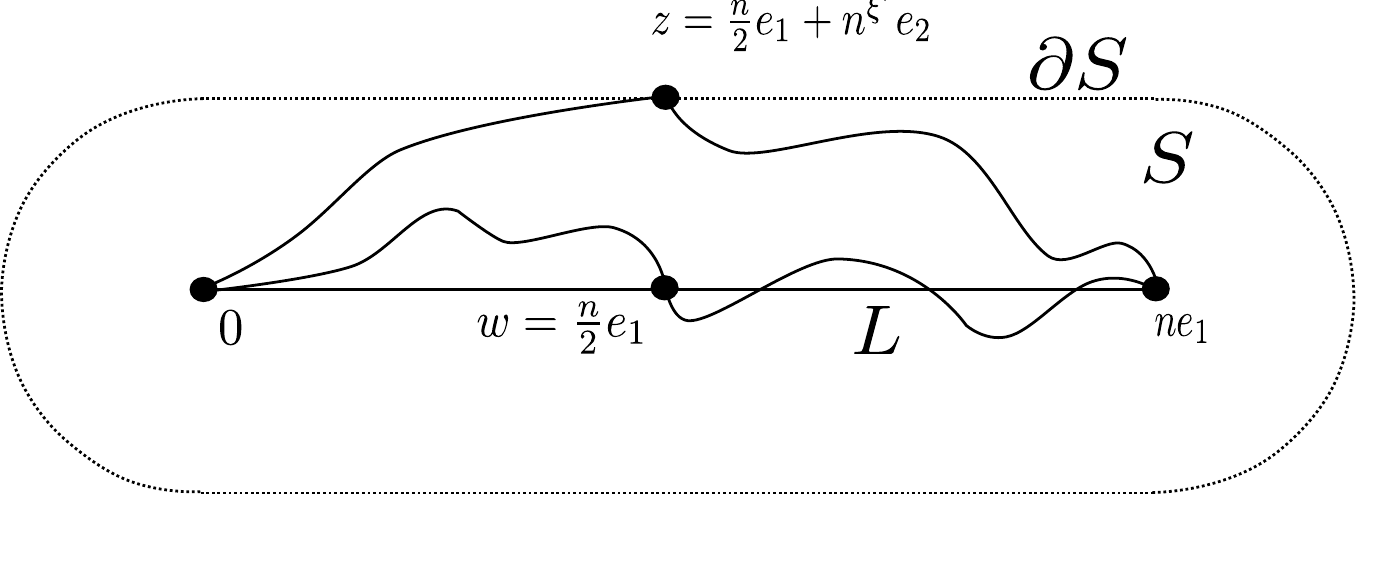}
\caption{Illustration of the argument for the inequality $\chi \geq 2\xi-1$. $S$ is the set of vertices within distance $n^{\xi'}$ of $L$, the line segment connecting $0$ and $ne_1$, and its boundary $\partial S$ contains $z$. The idea is to show that with high probability, the passage time from $0$ to $ne_1$ through $z$ is larger than the passage time through $(n/2)e_1$. When this occurs, $z$ cannot be on a geodesic from $0$ to $ne_1$.}
\label{fig: fig_3}
\end{figure}

We will not show \eqref{eq: wandering_bound} for all $z$, but only for $z = (n/2)e_1 + n^{\xi'}e_2$, which may not be an integer point, but we will pretend it is, to avoid notation. In this case, if $A_z$ occurs, we write
\begin{align}
0 &\geq T(0,ne_1) - T(0,(n/2)e_1) - T((n/2)e_1,ne_1) \nonumber \\
&= \left[ T(0,z) - T(0,(n/2)e_1) \right] + \left[ T(z,ne_1)-T((n/2)e_1,ne_1) \right] \label{eq: taco_1}.
\end{align}
(See Figure~\ref{fig: fig_3}.) These bracketed terms have the same distribution, so we will only bound one of them. Write $w=(n/2)e_1$ and estimate
\begin{equation}\label{eq: taco_taco}
T(0,z) - T(0,w) = g(z)-g(w) + \left[ T(0,z) - g(z) \right] + \left[ T(0,w) - g(w) \right].
\end{equation}
To estimate $g(z)-g(w)$, we write it as
\[
g\left( (n/2)e_1+n^{\xi'}e_2 \right) - g((n/2)e_1) = \frac{n}{2} \left[ g\left( e_1 + \frac{2n^{\xi'}}{n} e_2 \right) - g(e_1) \right].
\]
Here we must use information about the limit shape, namely that in certain directions it is known to have 
\index{curvature}%
``positive curvature.'' It is known (see \cite{Sutav}) that there exists a point on the limit shape boundary near which the boundary is locally positively curved, and instead of doing the argument in that direction, we assume this condition in direction $e_1$. It amounts to the statement that there exists $\delta,C_1>0$ such that if $v$ is a vector with $\|v\| < \delta$ and $v \perp e_1$, then
\[
g(e_1 + v) - g(e_1) \geq C_1\|v\|^2.
\]
We apply this in the above equation with $v = \frac{2n^{\xi'}}{n} e_2$. This is less than $\delta$ in norm for $n$ large because we may assume that $\xi' < 1$. If not, then $\xi \geq 1$, and in this case, the KPZ inequality reads $\chi \geq 1$, which is false, as known exponential concentration bounds imply that $\chi \leq 1/2$. So we obtain
\[
g(z) - g(w) \geq C_1 \frac{n}{2} \left( \frac{2n^{\xi'}}{n} \right)^2 = C_2 n^{2\xi' -1}.
\]
Returning to \eqref{eq: taco_taco}, we obtain
\[
T(0,z) - T(0,w) \geq C_2n^{2\xi'-1} + [T(0,z) - g(z)] + [T(0,w)-g(w)].
\]

The same development works for the other term of \eqref{eq: taco_1}, and we find that if $A_z$ occurs, then
\begin{align*}
0 
&\geq [T(0,z) - g(z)] + [T(0,w) - g(w)] + [T(z,ne_1)-g(ne_1-z)]\\
&\qquad+ [T(w,ne_1) - g(ne_1-w)] + 2C_2n^{2\xi'-1}.
\end{align*}
This means that at least one of the four bracketed terms is at least $(C_2/2) n^{2\xi'-1}$ in absolute value. Thus using symmetry,
\[
\mathbb{P}(A_z) \leq 2\mathbb{P}(|T(0,z) - g(z)| \geq (C_2/2) n^{2\xi'-1}) + 2\mathbb{P}(|T(0,w) - g(w)| \geq (C_2/2)n^{2\xi'-1}).
\]
By our choice of $\xi'$, we have $2\xi'-1 > \chi$. We have assumed exponential concentration above scale $\chi'$ (see \eqref{eq: time_concentration}), so these probabilities are (stretched) exponentially small in $n$. In other words, each one is smaller than $e^{-n^c}$ for some $c>0$. This shows \eqref{eq: wandering_bound} and completes the sketch of the bound $\chi \geq 2\xi-1$.


We turn to the other inequality, $\chi \leq 2\xi-1$. For technical reasons, we assume that $\chi>0$; the other case can be proved using a different argument \cite{Sutav}. Suppose it is false and chose $\chi', \chi'', \xi'$ such that
\begin{equation}\label{eq: chi_prime}
2\xi-1 < 2\xi'-1 < \chi' < \chi < \chi''.
\end{equation}
We define the variable $\delta T$, which first appeared in Licea-Newman-Piza \cite{LNP}: $\delta T = T - T'$, where
\[
T = T(0,ne_1) \text{ and } T' = T(n^{\xi'}e_2, ne_1 + n^{\xi'}e_2).
\]
Because $\xi' > \xi$, these passage times are nearly independent, as they are with high probability equal to two passage times restricted to disjoint sets of edges (``tubes'' of width $n^{\xi'}/2$ centered on the straight lines connecting their endpoints). Using the exponential concentration assumption on $D(0,ne_1)$ (from \eqref{eq: geo_concentration}), we then obtain if $T''$ is an independent copy of $T$,
\[
\mathrm{Var}~T = \frac{1}{2} \mathbb{E}(T-T'')^2 \leq C_3 \mathbb{E}(T-T')^2.
\]
Since $\chi' < \chi$, we find
\begin{equation}\label{eq: half_inequality}
n^{2\chi'} \leq C_4 \mathbb{E}(T-T')^2.
\end{equation}

\begin{figure}[!ht]
\centering
\includegraphics[width=4in,trim={.5cm 16cm 4cm 2cm},clip]{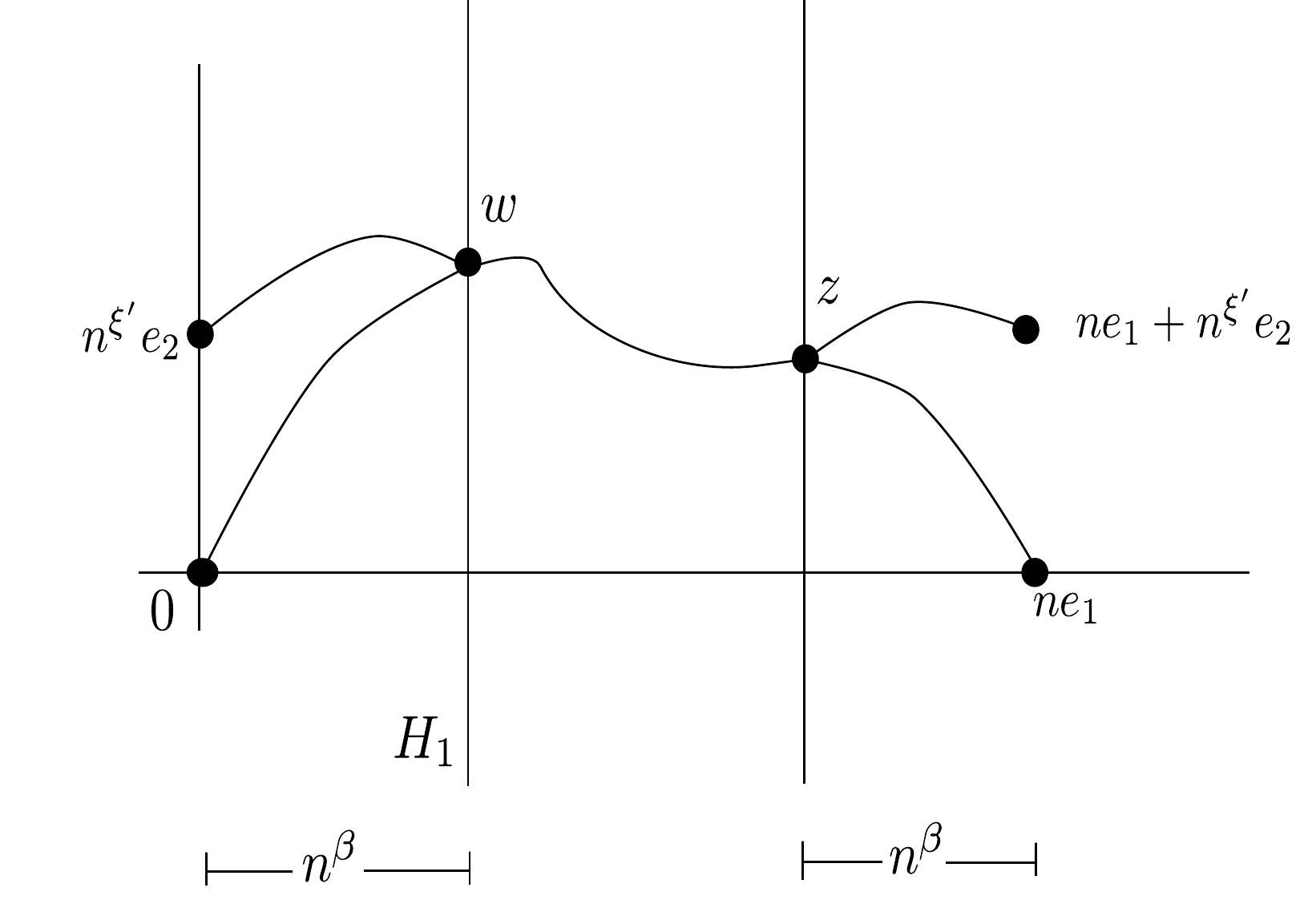}
\caption{Illustration of the argument for the inequality $\chi \leq 2\xi-1$. The path that connects $n^{\xi'}e_2$ to $ne_1 + n^{\xi'}e_2$ through $w$ (its first intersection with $H_1$) and $z$ is a geodesic. The path that connects $0$ to $ne_1$ through $w$ and $z$ is a possibly suboptimal path, and this produces inequality \eqref{eq: inequality_oh}.}
\label{fig: fig_4}
\end{figure}

The next step is to upper bound $(T-T')^2$. Let $w$ be the first intersection of a geodesic from $n^{\xi'} e_2$ to $n^{\xi'}e_2 + ne_1$ with the hyperplane $H_1 = \{y : y \cdot e_1 = n^{\beta}\}$ (for some $\beta < 1$ but very close to 1), and let $z$ be the first intersection of this geodesic with the hyperplane $\{y : y \cdot e_1 = n-n^{\beta}\}$. Note that
\begin{align*}
T = T(0,ne_1) &\leq T(0,w) + T(w,z) + T(z,ne_1) \\
&= T(0,w) - T(n^{\xi'}e_2,w) + T(n^{\xi'}e_2, ne_1 + n^{\xi'}e_2)\\
&\qquad + T(z,ne_1) - T(z,ne_1+n^{\xi'}e_2),
\end{align*}
so that
\begin{equation}\label{eq: inequality_oh}
T-T' \leq \left[ T(0,w) - T(n^{\xi'}e_2,w) \right] + \left[T(z,ne_1) - T(z,ne_1+n^{\xi'}e_2) \right].
\end{equation}
With (exponentially) high probability, $w$ and $z$ are not further than distance $C_5 n^{\xi'}$ from the $e_1$-axis, since $\xi' > \xi$ (see \eqref{eq: geo_concentration}). So using symmetry, we obtain with high probability
\[
T-T' \leq A+B,
\]
where $A$ and $B$ have the same distribution as
\[
D = \max\left\{ |T(0,w) - T(n^{\xi'}e_2,w)| : w \in H_1, \|w\| \leq C_5 n^{\xi'} \right\}.
\]
Reversing the roles of $T$ and $T'$, we obtain the same inequality for $T'-T$, and so
\[
\mathbb{E}(T-T')^2 \leq C_6\mathbb{E}D^2.
\]

We last have to bound $\mathbb{E}D^2$. This is a maximum over many different passage times (to all $w \in H_1$ with $\|w\| \leq C_5 n^{\xi'}$). However, since we have assumed that the fluctuation exponent $\chi$ exists in a strong sense (there is exponential concentration --- see \eqref{eq: time_concentration}), it is possible to replace this maximum with simply one passage time, and $\mathbb{E}D^2$ will increase by only a logarithmic factor. Thus we can write
\begin{equation}\label{eq: other_half_inequality}
\mathbb{E}(T-T')^2 \leq C_7 (\log n) \mathbb{E}\left[ T(0,w) - T(n^{\xi'}e_2,w) \right]^2,
\end{equation}
where $w = n^\beta e_1 + n^{\xi'}e_2$. As in the proof of the other inequality, we decompose this difference as
\begin{align*}
T(0,w) - T(n^{\xi'}e_2,w) &= [T(0,w) - \mathbb{E}T(0,w)] - [T(n^{\xi'}e_2,w) - \mathbb{E}T(n^{\xi'}e_2,w)] \\
&+ [\mathbb{E}T(0,w) - g(w)] - [\mathbb{E}T(n^{\xi'}e_2,w) - g(w-n^{\xi'}e_2)] \\
&+ g(w) - g(w-n^{\xi'}e_2).
\end{align*}
And once again, our exponential concentration assumption \eqref{eq: time_concentration} allows us to upper bound all the terms in the first two lines (with high probability) by $\|w\|^{\chi''}$. If $\xi \geq 1$, then our main inequality $\chi \leq 2\xi-1$ is simply $\chi \leq 2\xi -1$, where $2\xi-1$ is $\geq 1$, and we already know this to be true (as $\chi \leq 1/2$), so we can assume that $\xi < 1$. In this case, we can also enforce
\begin{equation}\label{eq: xi_prime}
\xi < \xi' < \beta < 1,
\end{equation}
and we obtain that $\|w\| = \|n^\beta e_1 + n^{\xi'}e_2\| \leq C_8 n^\beta$. Therefore from \eqref{eq: half_inequality} and \eqref{eq: other_half_inequality},
\begin{equation}\label{eq: equality_so_far}
n^{2\chi'} \leq C_9 (\log n) \left( C_{10}n^{\beta \chi''} + g(w) - g\left( w-n^{\xi'} e_2 \right)\right)^2.
\end{equation}
Again, we analyze the difference in $g$ by mandating a curvature assumption. We calculate
\begin{align*}
g(w) + g(w-n^{\xi'}e_2) &= g\left( n^\beta e_1 + n^{\xi'}e_2\right) - g(n^\beta e_1) \\
&= n^\beta \left( g\left( e_1 + n^{\xi'-\beta}e_2\right) - g(e_1) \right).
\end{align*}
Our curvature condition here is the opposite as in the previous inequality $\chi \leq 2\xi-1$. That is, we assume that there are $C_{11},\delta>0$ such that if $u$ satisfies $\|u\| < \delta$ and $u \perp e_1$, then
\[
g(e_1 + u) - g(e_1) \leq C_{11}\|u\|^2.
\]
Fortunately since the limit shape is convex, one can show that this inequality holds in almost every direction, so we will assume it in the $e_1$ direction, as it is written. Since $\xi' < \beta$, the term $n^{\xi'-\beta} < \delta$ for large $n$, and we obtain
\[
g(w) - g(w-n^{\xi'}e_2) \leq C_{11} n^\beta n^{2\xi'-2\beta'} = C_{11}n^{2\xi'-\beta}.
\]
Last we plug this back into \eqref{eq: equality_so_far} for
\[
n^{2\chi'} \leq C_9 (\log n) \left( C_{10} n^{2\beta \chi''} + C_{11} n^{2(2\xi'-\beta)} \right).
\]
This is true for all $n$ large, so
\[
2\chi' \leq \max\left\{ 2\beta \chi'', 2(2\xi'-\beta)\right\}.
\]
This holds for all $\beta, \chi', \chi'', \xi'$ satisfying \eqref{eq: chi_prime} and \eqref{eq: xi_prime}.

So take $\chi'' \downarrow \chi$ and $\chi' \uparrow \chi$ for fixed $\beta, \xi'$ for
\[
2\chi \leq \max\{2\beta \chi, 2(2\xi' -\beta)\}.
\]
As $\beta< 1$, we find
\[
\chi \leq 2\xi'-\beta.
\]
Now take $\beta \uparrow 1$ and $\xi' \downarrow \xi$ to obtain $\chi \leq 2\xi - 1$.
\end{proof}

\bigskip
\noindent
{\bf Acknowledgements.} I thank F.\ Rassoul-Agha for discussions about the subadditive ergodic theorem, and for the notes on the proof from his course that appeared in this article. 

\bibliographystyle{amsplain}

\end{document}